\documentclass[11pt,a4paper]{article}

\usepackage{amsmath}
\usepackage{amsthm}
\usepackage{amssymb}
\usepackage{cain_arxiv}
\usepackage{tikz}
\usetikzlibrary{matrix}

\theoremstyle{definition}
\newtheorem{definition}{Definition}[section]

\newtheorem{remark}[definition]{Remark}
\newtheorem{algorithm}[definition]{Algorithm}

\theoremstyle{plain}
\newtheorem{proposition}[definition]{Proposition}
\newtheorem{lemma}[definition]{Lemma}
\newtheorem{theorem}[definition]{Theorem}

\numberwithin{equation}{section}

\def\fullref#1#2{%
  \ifdefined\hyperref%
    {\hyperref[#2]{#1\space\penalty 200\relax\ref*{#2}}}%
  \else%
    {#1\space\penalty 200\relax\ref{#2}}%
  \fi%
}

\newcommand{\defterm}[1]{\textit{#1}}

\newcommand{\pres}[2]{\left\langle #1\:|\:#2 \right\rangle}
\newcommand{\cgen}[1]{#1^{\#}}

\newcommand{\nset}{\mathbb{N}}

\newcommand{\emptyword}{\varepsilon}
\newcommand{\rel}[1]{\mathcal{#1}}

\newcommand{\imreduces}{\rightarrow}
\newcommand{\reduces}{\rightarrow^*}
\newcommand{\posreduces}{\rightarrow^+}

\newcommand{\thue}{\leftrightarrow^*}

\newcommand{\rpad}{\delta_{\mathrm{R}}}
\newcommand{\lpad}{\delta_{\mathrm{L}}}

\DeclareMathOperator{\csw}{CS}

\newcommand{\rev}{\mathrm{rev}}

\newcommand{\lex}{\mathrm{lex}}

\newcommand{\bst}{\mathcal{T}}

\begin{document}

\title{Rewriting systems and biautomatic structures for Chinese,
  hypoplactic, and sylvester monoids}

\author{Alan J. Cain, Robert D. Gray, Ant\'{o}nio Malheiro}

\thanks{During the research that led to this paper, the first author was supported by an Investigador {\sc FCT}
  fellowship ({\sc IF}/01622/2013/{\sc CP}1161/{\sc CT}0001). This work was developed within the research activities of
  the Centro de \'{A}lgebra da Universidade de Lisboa, {\sc FCT} project {\sc PEst-OE}/{\sc MAT}/{\sc UI}0143/2014, and
  of the Departamento de Matem\'{a}tica da Faculdade de Ci\^{e}ncias e Tecnologia da Universidade Nova de Lisboa.}

\date{}

\maketitle

\address[AJC]{%
Centro de Matem\'{a}tica e Aplica\c{c}\~{o}es, \\
Faculdade de Ci\^{e}ncias e Tecnologia, \\
Universidade Nova de Lisboa, 2829--516 Caparica, Portugal
}
\email{%
a.cain@fct.unl.pt
}
\webpage{%
www.fc.up.pt/pessoas/ajcain/
}

\address[RDG]{%
School of Mathematics, University of East Anglia, \\
Norwich NR4 7TJ, United Kingdom
}
\email{%
Robert.D.Gray@uea.ac.uk
}

\address[AM]{%
Departamento de Matem\'{a}tica, Faculdade de Ci\^{e}ncias e Tecnologia, \\
Universidade Nova de Lisboa, 2829--516 Caparica, Portugal \\
\null\quad and \\
Centro de \'{A}lgebra da Universidade de Lisboa, \\
Av. Prof. Gama Pinto 2, 1649--003 Lisboa, Portugal 
}
\email{%
ajm@fct.unl.pt
}

\begin{dedication}
\textit{Dedicated to Stuart W. Margolis on the occasion of his 60th birthday}
\end{dedication} 

\begin{abstract}
This paper studies complete rewriting systems and biautomaticity for
three interesting classes of finite-rank homogeneous mon\-oids: Chinese
monoids, hypoplactic monoids, and sylvester mon\-oids. For Chinese
monoids, we first give new presentations via finite complete rewriting
systems, using more lucid constructions and proofs than those given
independently by Chen \& Qui and G\"{u}zel Karpuz; we then construct
biautomatic structures. For hypoplactic monoids, we construct finite
complete rewriting systems and biautomatic structures. For sylvester
monoids, which are not finitely presented, we prove that the standard
presentation is an infinite complete rewriting system, and construct
biautomatic structures. Consequently, the monoid algebras
corresponding to monoids of these classes are automaton algebras in
the sense of Ufnarovskij.
\end{abstract}

%\tableofcontents

\section{Introduction}

The aim of this paper is to study whether certain homogeneous monoids
admit presentations via finite complete rewriting systems or are
biautomatic. The focus is on Chinese monoids, hypoplactic monoids, and
sylvester monoids of finite rank. All three classes of monoids are
related to Plactic monoids. In a previous paper \cite{cgm_plactic}, we
answered a question of Zelmanov by constructing biautomatic structures
and presentations via finite complete rewriting systems for Plactic
monoids of finite rank. The present paper is partly motivated by that
earlier work, but the techniques we use here are original. As we
discuss below, our results also have consequences for the study of the
corresponding monoid algebras.

The Chinese monoid was introduced by Duchamp \& Krob
\cite{duchamp_plactic}, as one of the (few) multi-homogeneous monoids
with the same multihomogeneous growth as the Plactic monoid. Cassaigne
et al.~\cite{cassaigne_chinese} made the first fundamental study of
the Chinese monoid, and the Chinese algebra has also been studied
\cite{chen_chinese,jaszunska_chineserank3,jaszunska_chinese}. The
result that finite-rank Chinese monoids are presented by finite
complete rewriting systems was obtained in the context of
Gr\"{o}bner--Shirshov bases by Chen \& Qiu \cite{chen_grobner}, and
later in the context of rewriting systems by G\"{u}zel Karpuz
\cite{guzelkarpuz_fcrs}. In both approaches the standard presentation
of the Chinese monoid is the starting point. Chen \& Qiu apply the
Shirshov algorithm; G\"{u}zel Karpuz applies the Knuth--Bendix
completion procedure. These procedures parallel each other
\cite{heyworth_rewriting} and consist of adding rewriting rules which
arise from the analysis of all possible overlaps. In both papers,
proving confluence of the resulting rewriting system by checking the
37 possible critical pairs is left as an exercise for the interested
reader. In \fullref{\S}{subsec:chinesefcrs}, we use different
generating sets to construct presentations via a finite complete
rewriting systems for Chinese monoids: we think these rewriting
systems are easier to understand and the proofs more elegant. We then
proceed to prove that finite-rank Chinese monoids are biautomatic in
\fullref{\S}{subsec:chinesebiauto}, exhibiting en route a left-handed
analogue of the algorithm of Cassaigne et
al.~\cite[\S~2.2]{cassaigne_chinese} for right-multiplying an element
of the Chinese monoid by a generator.

The hypoplactic algebra was introduced by Krob \& Thibon
\cite{krob_noncommutative} as a quotient of the Plactic algebra. The
fundamental study of the underlying hypoplactic monoid, which is a
quotient of the Plactic monoid, is due to Novelli
\cite{novelli_hypoplactic}. In \fullref{\S}{subsec:hypoplacticfcrs},
we give a neat construction and proof that finite-rank hypoplactic
monoids admit presentations via finite complete rewriting systems. In
\fullref{\S}{subsec:hypoplacticbiauto}, we prove that hypoplactic
monoids of finite rank are biautomatic.

The sylvester monoid was defined by Hivert, Novelli \&
Thibon~\cite{hivert_algebra} as an analogue of the plactic monoid
where Schensted's algorithm for insertion into Young tableaux (see
\cite[ch.~5]{lothaire_algebraic}) is replaced by insertion into a
binary search tree; from the sylvester monoid, they then recover the
Hopf algebra of planar binary trees defined by Loday \& Ronco
\cite{loday_hopf}. Finite-rank sylvester monoids are not finitely
presented, but in \fullref{\S}{subsec:sylvestercrs}, we show that the
standard presentations for finite-rank sylvester monoids form
(infinite) complete rewriting systems. In
\fullref{\S}{subsec:sylvesterbiautomatic}, we proceed to show that
sylvester monoids of finite rank are biautomatic.

The existence of finite complete rewriting systems for the Chinese and
hypoplactic monoids immediately implies the existence of finite
Gr\"{o}bner--Shirshov bases for the Chinese and hypoplactic algebras
\cite{heyworth_rewriting}. From the biautomaticity of the Chinese,
hypoplactic, and sylvester monoids, we immediately recover the
solvability of the word problem in quadratic time
\cite[Corollary~3.7]{campbell_autsg}. Furthermore, the biautomaticity
of these monoids implies that each admits a regular cross-section
\cite[Corollary~5.6]{campbell_autsg}, which in turn implies that the
corresponding monoid algebras are automaton algebras in the sense of
Ufarnovskij \cite{ufnarovskij_combinatorial}.

In related work \cite{cgm_homogeneous}, we give examples of
homogeneous and multihomogeneous monoids that do not admit finite
complete rewriting systems or biautomatic structures, and indeed we
show that the two notions are independent within the classes of
homogeneous and multihomogeneous monoids. Thus the results in this
paper are not simply consequences of more general results for
homogeneous or multi-homogeneous monoids.

Since Chinese, hypolactic, and sylvester monoids are biautomatic, they
have decidable conjugacy problem. (To be precise, one can decide the
$o$-conjugacy relation, define by Otto~\cite{otto_conjugacy}, using
reasoning similar to the group
case~\cite[Theorem~2.5.7]{epstein_wordproc}.)  The algorithm for
biautomatic monoids is exponential-time in general. An interesting
open question is whether one can improve this exponential bound for
\emph{homogeneous} biautomatic monoids. (However, there are
(non-biautomatic) homogeneous and even multihomogeneous monoids in
which conjugacy is undecidable \cite[Theorem~4.1]{cm_conjugacy}.)

\section{Preliminaries}

\subsection{Words and presentations}

We denote the empty word (over any alphabet) by $\emptyword$. For an
alphabet $A$, we denote by $A^*$ the set of all words over $A$. When
$A$ is a generating set for a monoid $M$, every element of $A^*$ can
be interpreted either as a word or as an element of $M$. For words
$u,v \in A^*$, we write $u=v$ to indicate that $u$ and $v$ are equal
as words and $u=_M v$ to denote that $u$ and $v$ represent the same
element of the monoid $M$. The length of $u \in A^*$ is denoted $|u|$,
and, for any $a \in A$, the number of symbols $a$ in $u$ is denoted
$|u|_a$.

For any relation $\rel{R}$ on $A^*$, the presentation
$\pres{A}{\rel{R}}$ defines [any monoid isomorphic to]
$A^*/\cgen{\rel{R}}$ , where $\cgen{\rel{R}}$ denotes the congruence
generated by $\rel{R}$.  The presentation $\pres{A}{\rel{R}}$ is
\defterm{homogeneous} (respectively, \defterm{multi-homogeneous}) if
for every $(u,v) \in \rel{R}$ and $a \in A$, we have $|u| = |v|$
(respectively, $|u|_a = |v|_a$). That is, in a homogeneous
presentation, defining relations preserve length; in a
multi-homogenous presentation, defining relations preserve the
numbers of each symbol. A monoid is \defterm{homogeneous}
(respectively, \defterm{multi-homogeneous}) if it admits a
\defterm{homogeneous} (respectively, \defterm{multi-homogeneous})
presentation.

Any total order $\leq$ on an alphabet $A$ induces a total order
$\leq_\lex$ on $A^*$, where $w \leq_\lex w'$ if and only if either $w$
is proper prefix of $w'$ or if $w=paq$, $w'=pbr$ and $a \leq b$ for
some $p,q,r \in A^*$, and $a,b \in A$. The order $\leq_\lex$ is the
\defterm{lexicographic order induced by $\leq$}. Notice that
$\leq_\lex$ is not a well-order, but that it is left compatible with
concatenation.

\subsection{String rewriting systems}
\label{subsec:srs}

A \defterm{string rewriting system}, or simply a \defterm{rewriting
  system}, is a pair $(A,\rel{R})$, where $A$ is a finite alphabet and
$\rel{R}$ is a set of pairs $(\ell,r)$, usually written $\ell
\imreduces r$, known as \defterm{rewriting rules} or simply
\defterm{rules}, drawn from $A^* \times A^*$. The single reduction
relation $\imreduces_{\rel{R}}$ is defined as follows: $u
\imreduces_{\rel{R}} v$ (where $u,v \in A^*$) if there exists a
rewriting rule $(\ell,r) \in \rel{R}$ and words $x,y \in A^*$ such
that $u = x\ell y$ and $v = xry$. That is, $u \imreduces_{\rel{R}} v$
if one can obtain $v$ from $u$ by substituting the word $r$ for a
subword $\ell$ of $u$, where $\ell \imreduces r$ is a rewriting
rule. The reduction relation $\reduces_{\rel{R}}$ is the reflexive and
transitive closure of $\imreduces_{\rel{R}}$. The process of replacing
a subword $\ell$ by a word $r$, where $\ell \imreduces r$ is a rule,
is called \defterm{reduction} by application of the rule $\ell
\imreduces r$; the iteration of this process is also called
reduction. A word $w \in A^*$ is \defterm{reducible} if it contains a
subword $\ell$ that forms the left-hand side of a rewriting rule in
$\rel{R}$; it is otherwise called \defterm{irreducible}.

The rewriting system $(A,\rel{R})$ is \defterm{finite} if both $A$ and
$\rel{R}$ are finite. The rewriting system $(A,\rel{R})$ is
\defterm{noetherian} if there is no infinite sequence $u_1,u_2,\ldots
\in A^*$ such that $u_i \imreduces_{\rel{R}} u_{i+1}$ for all $i \in
\nset$. That is, $(A,\rel{R})$ is noetherian if any process of
reduction must eventually terminate with an irreducible word. The
rewriting system $(A,\rel{R})$ is \defterm{confluent} if, for any
words $u, u',u'' \in A^*$ with $u \reduces_{\rel{R}} u'$ and $u
\reduces_{\rel{R}} u''$, there exists a word $v \in A^*$ such that $u'
\reduces_{\rel{R}} v$ and $u'' \reduces_{\rel{R}} v$. A rewriting
system that is both confluent and noetherian is \defterm{complete}. If
$(A,\rel{R})$ is a complete rewriting system, then for every word $u$
there is a unique irreducible word $w$ such that $u \reduces_{\rel{R}}
w$; this word is called the \defterm{normal form} of $u$.

The rewriting system $(A,\rel{R})$ is \defterm{globally finite} if,
for each $w \in A^*$, there are only finitely many words $w'$ such
that $w \reduces_{\rel{R}} w'$. It is \defterm{acyclic} if there is no
word $w$ such that $w \posreduces_{\rel{R}} w$. An acyclic globally
finite rewriting system is noetherian \cite[Lemma
  2.2.5]{baader_termrewriting}.

The \defterm{Thue congruence} $\thue_{\rel{R}}$ is the equivalence
relation generated by $\imreduces_{\rel{R}}$. The elements of the
monoid presented by $\pres{A}{\rel{R}}$ are the
$\thue_{\rel{R}}$-equivalence classes. If $(A,\rel{R})$ is complete,
then the language of normal form words forms a cross-section of the
monoid: that is, each element of the monoid presented by
$\pres{A}{\rel{R}}$ has a unique normal form representive.

\subsection{Automaticity and biautomaticity}

This subsection contains the definitions and basic results from the
theory of automatic and biautomatic monoids needed hereafter. For
further information on automatic semigroups,
see~\cite{campbell_autsg}. We assume familiarity with basic notions of
automata and regular languages (see, for example,
\cite{hopcroft_automata}) and transducers and rational relations (see,
for example, \cite{berstel_transductions}).

\begin{definition}
Let $A$ be an alphabet and let $\$$ be a new symbol not in $A$. Define
the mapping $\rpad : A^* \times A^* \to ((A\cup\{\$\})\times (A\cup
\{\$\}))^*$ by
\[
(u_1\cdots u_m,v_1\cdots v_n) \mapsto
\begin{cases}
(u_1,v_1)\cdots(u_m,v_n) & \text{if }m=n,\\
(u_1,v_1)\cdots(u_n,v_n)(u_{n+1},\$)\cdots(u_m,\$) & \text{if }m>n,\\
(u_1,v_1)\cdots(u_m,v_m)(\$,v_{m+1})\cdots(\$,v_n) & \text{if }m<n,
\end{cases}
\]
and the mapping $\lpad : A^* \times A^* \to ((A\cup\{\$\})\times (A\cup \{\$\}))^*$ by
\[
(u_1\cdots u_m,v_1\cdots v_n) \mapsto
\begin{cases}
(u_1,v_1)\cdots(u_m,v_n) & \text{if }m=n,\\
(u_1,\$)\cdots(u_{m-n},\$)(u_{m-n+1},v_1)\cdots(u_m,v_n) & \text{if }m>n,\\
(\$,v_1)\cdots(\$,v_{n-m})(u_1,v_{n-m+1})\cdots(u_m,v_n) & \text{if }m<n,
\end{cases}
\]
where $u_i,v_i \in A$.
\end{definition}

\begin{definition}
\label{def:autstruct}
Let $M$ be a monoid. Let $A$ be a finite alphabet representing a set
of generators for $M$ and let $L \subseteq A^*$ be a regular language such
that every element of $M$ has at least one representative in $L$.  For
each $a \in A \cup \{\emptyword\}$, define the relations
\begin{align*}
L_a &= \{(u,v): u,v \in L, {ua} =_M {v}\}\\
{}_aL &= \{(u,v) : u,v \in L, {au} =_M {v}\}.
\end{align*}
The pair $(A,L)$ is an \defterm{automatic structure} for $M$ if
$L_a\rpad$ is a regular languages over $(A\cup\{\$\}) \times
(A\cup\{\$\})$ for all $a \in A \cup \{\emptyword\}$. A monoid $M$ is
\defterm{automatic} if it admits an automatic structure with respect to
some generating set.

The pair $(A,L)$ is a \defterm{biautomatic structure} for $M$ if
$L_a\rpad$, ${}_aL\rpad$, $L_a\lpad$, and ${}_aL\lpad$ are regular
languages over $(A\cup\{\$\}) \times (A\cup\{\$\})$ for all $a \in A
\cup \{\emptyword\}$. A monoid $M$ is \defterm{biautomatic} if it
admits a biautomatic structure with respect to some generating
set. [Note that biautomaticity implies automaticity.]
\end{definition}

Unlike the situation for groups, biautomaticity for monoids and
semigroups, like automaticity, is dependent on the choice of
generating set \cite[Example~4.5]{campbell_autsg}. However, for
monoids, biautomaticity and automaticity are independent of the choice
of \emph{semigroup} generating sets \cite[Theorem~1.1]{duncan_change}.

Hoffmann \& Thomas have made a careful study of biautomaticity for
semigroups \cite{hoffmann_biautomatic}. They distinguish four notions
of biautomaticity for semigroups:
\begin{itemize}

\item \defterm{right-biautomaticity}, where $L_a\rpad$ and ${}_aL\rpad$ are
  regular languages;

\item \defterm{left-biautomaticity}, where $L_a\lpad$ and ${}_aL\lpad$ are
  regular languages;

\item \defterm{same-biautomaticity}, where $L_a\rpad$ and ${}_aL\lpad$ are
  regular languages;

\item \defterm{cross-biautomaticity}, where ${}_aL\rpad$ and $L_a\lpad$ are
  regular languages.

\end{itemize}
These notions are all equivalent
for groups and more generally for cancellative semigroups
\cite[Theorem~1]{hoffmann_biautomatic} but distinct for semigroups
\cite[Remark~1 \& \S~4]{hoffmann_biautomatic}. In the sense used in
this paper, `biautomaticity' implies \emph{all four} notions of
biautomaticity above.

%% \begin{proposition}[{\cite[Proposition~4]{hoffmann_biautomatic}}]
%% Let $R \subseteq A^∗ \times A^∗$ and suppose there is a constant $t$
%% such that
%% \[
%% (u,v) \in R \implies \bigl||u|-|v|\bigr| \leq k;
%% \]
%% then $R\rpad$ is regular if and only if $R\lpad$ is regular.
%% \end{proposition}

%% Let $Q,R\subseteq A^* \times A^*$ be relations. Define
%% \[
%% Q \circ R = \{(u,v) : (\exists w \in A^+)((u,w) \in Q \land (w,v) \in R)\},
%% \]
%% and
%% \[
%% Q^{-1} = \{(u,v) : (v,u) \in Q\}.
%% \]

%% \begin{proposition}[{\cite[Proposition~2.3]{campbell_autsg}}]
%% \label{prop:composition}
%% Let $Q,R\subseteq A^* \times A^*$ be such that $Q\rpad$ and $R\rpad$
%% are regular. Then $(Q \circ R)\rpad$ and $Q^{-1}\rpad$ are also
%% regular.
%% \end{proposition}

\subsection{Rational relations}

In proving that $R\rpad$ or $R\lpad$ is regular, where $R$ is a
relation on $A^*$, a useful strategy is to prove that $R$ is a
rational relation (that is, a relation recognized by a finite
transducer \cite[Theorem~6.1]{berstel_transductions}) and then apply
the following result, which is a combination of
\cite[Corollary~2.5]{frougny_synchronized} and
\cite[Proposition~4]{hoffmann_biautomatic}:

\begin{proposition}
\label{prop:rationalbounded}
If $R \subseteq A^* \times A^*$ is rational relation and there is a
constant $k$ such that $\bigl||u|-|v|\bigr| \leq k$ for all $(u,v) \in
R$, then $R\rpad$ and $R\lpad$ are regular.
\end{proposition}

\begin{remark}
When constructing transducers to recognize particular relations, we
will make use of certain strategies.

One strategy will be to consider a transducer reading elements of a
relation $R$ from \emph{right to left}, instead of (as usual) left to
right. In effect, such a transducer recognizes the reverse of $R$,
which is the relation
\[
R^\rev = \{(u^\rev,v^\rev) : (u,v) \in R\},
\] 
where $u^\rev$ and $v^\rev$ are the reverses of the words $u$ and $v$
respectively. Since the class of rational relations is closed under
reversal \cite[p.65--66]{berstel_transductions}, constructing such a
(right-to-left) transducer suffices to show that $R$ is a rational
relation.

Another important strategy will be for the transducer to
non-det\-er\-min\-is\-tic\-al\-ly guess some symbol yet to be read. More exactly,
the transducer will non-deterministically select a symbol and store it
in its state. When it later reads the relevant symbol, it checks it
against the stored guessed symbol. If the guess was correct, the
transducer continues. If the guess was wrong, the transducer enters a
failure state. Similarly, the transducer can non-deterministically
guess that it has reached the end of its input and enter an accept
state. If it subsequently reads another symbol, it knows that its
guess was wrong, and it enters a failure state.
\end{remark}

\section{Chinese monoid}

\subsection{Staircases}

Let $n \in \nset$. Let $A$ be the finite ordered alphabet $\{1 < 2 <
\ldots < n\}$. Let $\rel{R}$ be the set of defining relations
\begin{align}
\bigl\{(zyx,{}&zxy), \label{eq:chineserela} \\
  (zxy,{}&yzx) : x \leq y\leq z\bigr\}.\label{eq:chineserelb}
\end{align}
Then the \defterm{Chinese monoid of rank $n$}, denoted $C_n$, is presented by
$\pres{A}{\rel{R}}$.

Cassaigne et al.~\cite[\S~2]{cassaigne_chinese} give a set of normal
forms for elements of the Chinese monoid. They point out that every
element has a unique representative of the form
\begin{equation}
\ell^{(1)} \ell^{(2)} \ell^{(3)}\cdots \ell^{(n)} \label{eq:chinesenf}
\end{equation}
with
\begin{equation}
\label{eq:chineserow}
\ell^{(k)} = (k1)^{\sigma_{k1}}(k2)^{\sigma_{k2}}\cdots (k(k-1))^{\sigma_{k(k-1)}}k^{\sigma_k},
\end{equation}
where the exponents $\sigma_{kj}$ and $\sigma_k$ lie in $\nset^0$. (Notice that in
\eqref{eq:chineserow}, $k-1$ is a single symbol.)  Cassaige et
al.\ arrange the exponents $\sigma_{kj}$ and $\sigma_k$ in a
\defterm{Chinese staircase}, which is an analogue of the planar
representation of a tableau for the Plactic monoid (see
\cite[ch.~5]{lothaire_algebraic}). For example, in the case $n=4$ the
$\sigma_{ki}$ and $\sigma_k$ are arranged as follows:

\medskip
{\centering
\begin{tikzpicture}
\matrix (staircase) [
matrix of math nodes,
nodes={rectangle,draw,minimum width=9mm,minimum height=9mm},
row sep={between borders,-\pgflinewidth},
column sep={between borders,-\pgflinewidth},
]
{
          &             &             & \sigma_{1} \\
          &             & \sigma_{2}  & \sigma_{21} \\
          & \sigma_{3}  & \sigma_{32} & \sigma_{31} \\
 \sigma_4 & \sigma_{43} & \sigma_{42} & \sigma_{41} \\
};
\path (staircase-1-4.east) node[anchor=west] {$1$};
\path (staircase-2-4.east) node[anchor=west] {$2$};
\path (staircase-3-4.east) node[anchor=west] {$3$};
\path (staircase-4-4.east) node[anchor=west] {$4$};
\path (staircase-4-1.south) node[anchor=north] {$4$};
\path (staircase-4-2.south) node[anchor=north] {$3$};
\path (staircase-4-3.south) node[anchor=north] {$2$};
\path (staircase-4-4.south) node[anchor=north] {$1$};
\end{tikzpicture}
\par}
\medskip

Notice that rows are indexed from top to bottom and columns from right
to left.  Because the $k$-th row of the staircase contains the
exponents for the word $\ell^{(k)}$, each such word $\ell^{(k)}$ (see
\eqref{eq:chineserow}) is called a \defterm{Chinese row}. A normal
form word \eqref{eq:chinesenf} is called a \defterm{Chinese staircase
  word}.

\subsection{Finite complete rewriting system}
\label{subsec:chinesefcrs}

As noted in the introduction, Chen \& Qiu \cite{chen_grobner} and
G\"{u}zel Karpuz \cite{guzelkarpuz_fcrs} independently constructed
finite complete rewriting systems for finite-rank Chinese
monoids. However, proving confluence of these systems relies on
checking that all 37 possible critical pairs resolve. In this section
we present an alternative finite complete rewriting system. By
changing the generating set, we obtain a rewriting system whose
language of irreducible words corresponds closely to the Chinese
staircase words. Once we prove that the rewriting system is
noetherian, we can deduce confluence quickly from the fact that
Chinese staircase words form a set of unique representatives for the
monoid.

Let \[D=\{d_{\alpha\beta}: \alpha, \beta\in A, \alpha >\beta \}\cup
\{d_{\alpha}: \alpha\in A\}.\] The idea is that symbols
$d_{\alpha\beta}$ and $d_\alpha$ represent, respectively, the elements
$\alpha \beta$ and $\alpha$ of $C_n$. Therefore the original
generating set $A$ is essentially included in this new set, and so $D$
also generates $C_n$.

Consider a Chinese staircase word \eqref{eq:chinesenf}. By replacing
each Chinese row $\ell^{(k)}$ (as in \eqref{eq:chineserow}) by
\begin{equation*}
d_{k1}^{\sigma_{k1}}d_{k2}^{\sigma_{k2}}\cdots d_{k(k-1)}^{\sigma_{k(k-1)}}d_k^{\sigma_k},
\end{equation*}
we obtain a word over the generating set $D$ representing the same
element of $C_n$; we call such a word a \defterm{$D$-Chinese staircase
  word}. Notice that we can start from a $D$-Chinese staircase word
and recover the original Chinese staircase word by simply reading off
the subscripts of the symbols in $D$. Thus there is a one-to-one
correspondence between Chinese staircase words and $D$-Chinese
staircase words, and so the set of $D$-Chinese staircase words is also
a language of unique normal forms for $C_n$.

Define a relation $\preceq$ on the alphabet $D$ as follows: for $d,d'
\in D$, we have $d \preceq d'$ if and only if $dd'$ is a $D$-Chinese
staircase word. As an immediate consequence of this definition, notice
that $d \preceq d'$ if and only if $d$ and $d'$ can appear in that
order, not necessarily as adjacent symbols, in a $D$-Chinese staircase
word. Thus $\preceq$ is a total order.

Now define a set of rewriting rules $\rel{T}$ on $D^*$ as follows:
\[
\rel{T}=\{dd'\rightarrow \csw_D(dd'): d \succ d'\},
\]
where $\csw_D(w)$ denotes the unique $D$-Chinese staircase word that
equals $dd'$ in $C_n$.  It is obvious from the definition that
$\rel{T}$ is finite and each rewriting rule in $\rel{T}$ holds in the
Chinese monoid $C_n$.

The next lemma is immediate from the definition of $\rel{T}$.

\begin{lemma}
\label{lem:irredchinese}
The irreducible words of the rewriting system $(D,\rel{T})$ are the
$D$-Chinese staircase words.
\end{lemma}

As a consequence we deduce that the presentation $\pres{D}{\rel{T}}$
defines the Chinese monoid $C_n$. In the next lemma we show that
$(D,\rel{T})$ is noetherian, and thus it will follow from the previous
lemma that $(D,\rel{T})$ is complete. In order to show termination of
$\rel{T}$ we will fully characterize the relation $\succ$ on
$D$. 

Let $\alpha,\beta,\gamma,\delta\in A$. From the definition of
$\preceq$ and the form of $D$-Chinese staircase words, we have:
\begin{align*}
d_\alpha \succ d_\beta &\iff \alpha >\beta;\\
d_{\alpha\beta}\succ d_{\gamma} &\iff \alpha >\gamma;\\
d_\alpha \succ d_{\beta\gamma} &\iff \alpha \geq \beta;\\
d_{\alpha\beta} \succ d_{\gamma\delta} &\iff (\alpha >\gamma) \vee (\alpha=\gamma \wedge \beta > \delta).
\end{align*}
Each of the left hand sides of a rewriting rule in $\rel{T}$
corresponds to one of the above cases. In the proof of the next lemma
we will see the resulting Chinese staircase words for each of the left
hand sides in $\rel{T}$.

\begin{lemma}
\label{lem:chinesenoetherian}
The rewriting system $\rel{T}$ is noetherian.
\end{lemma}

\begin{proof}
We will show that, for any $w,w'\in D^*$, if $w\imreduces_{\rel{T}}w'$
then $w \succ_\lex w'$. Since $\succ_\lex$ is left compatible with
concatenation, it suffices to show that $(w_1,w_2)\in {\rel{T}}$
implies $w_1\succ_\lex w_2$.

So let $(w_1,w_2)\in {\rel{T}}$. Depending on the form of $w_1$ we
divide the proof into four cases:
\begin{enumerate}

\item $w_1=d_\alpha d_\beta$ with $d_\alpha\succ d_\beta$.  Then
  $\alpha >\beta$. Therefore $\csw_D(d_\alpha d_\beta) = d_{\alpha
  \beta}$. We have $d_\alpha \succ d_{\alpha\beta}$ and so $d_\alpha
  d_\beta \succ_\lex d_{\alpha\beta}$.

\item $w_1=d_{\alpha\beta}d_\gamma$ with $d_{\alpha\beta}\succ d_\gamma$.
  Then $\alpha >\gamma$ (and $\alpha >\beta$ by the definition of
  $D$). Now, consider two sub-cases:
\begin{enumerate}
\item $\beta \geq \gamma $. Then
\begin{align*}
\alpha \beta \gamma &=_{C_n} \alpha \gamma \beta &&\text{(by \eqref{eq:chineserela})} \\
&=_{C_n} \beta \alpha\gamma && \text{(by \eqref{eq:chineserelb})}
\end{align*}
Therefore $\csw_D(d_{\alpha\beta} d_\gamma) = d_\beta d_{\alpha
  \gamma}$. Since $\alpha >\beta$, we have $d_{\alpha\beta}\succ
d_\beta$ and so $d_{\alpha\beta} d_\gamma \succ_\lex d_\beta d_{\alpha
  \gamma}$.

\item $\gamma >\beta $. Then $\alpha \beta \gamma =_{C_n} \gamma
  \alpha \beta$ by \eqref{eq:chineserelb}. Therefore
  $\csw_D(d_{\alpha\beta}d_\gamma) = d_\gamma d_{\alpha\beta}$. Since
  $\alpha >\gamma$, we have $d_{\alpha\beta}\succ d_\gamma$ and so
  $d_{\alpha\beta}d_\gamma \succ_\lex d_\gamma d_{\alpha\beta}$.
\end{enumerate}

\item If $w_1=d_\alpha d_{\beta\gamma}$ with $d_\alpha
  \succ d_{\beta\gamma}$.  Then $\alpha \geq \beta$ (and $\beta >\gamma$ by
  the definition of $D$). We have
\begin{align*}
\alpha \beta \gamma &=_{C_n} \alpha \gamma \beta && \text{(by \eqref{eq:chineserela})} \\
&=_{C_n} \beta \alpha \gamma. && \text{(by \eqref{eq:chineserelb})} 
\end{align*}
Now consider two sub-cases:
\begin{enumerate}
\item $\alpha =\beta$. Then $\csw_D(d_\alpha d_{\beta \gamma}) =
  d_{\alpha\gamma} d_{\beta}$. Since $\alpha > \gamma$, we have $d_{\alpha}
  \succ d_{\alpha\gamma}$ and so $d_\alpha d_{\beta \gamma} \succ_\lex
  d_{\alpha \gamma} d_\beta$.
\item $\alpha >\beta$. Then $\csw_D(d_\alpha d_{\beta \gamma}) =
  d_\beta d_{\alpha \gamma}$. Since $\alpha > \beta$, we have
  $d_\alpha \succ d_\beta$ and so $d_\alpha d_{\beta \gamma} \succ_\lex
  d_\beta d_{\alpha \gamma}$.
\end{enumerate}

\item If $w_1=d_{\alpha\beta}d_{\gamma\delta}$, with $d_{\alpha\beta}\succ
  d_{\gamma\delta}$. Then either $\alpha > \gamma $, or
  $\alpha=\gamma$ and $\beta > \delta$. (In both sub-cases, $\alpha
  >\beta$ and $\gamma >\delta$ by the definition of $D$.) Consider
  these sub-cases separately:
\begin{enumerate}
\item $\alpha =\gamma$ and $\beta > \delta$. Then $\alpha=\gamma >
  \beta >\delta$, so
\begin{align*}
\alpha {\beta \gamma \delta} &= \alpha\beta\alpha\delta \\
&=_{C_n} \alpha\alpha\beta\delta && \text{(by \eqref{eq:chineserela} applied to $\alpha\beta\alpha$)} \\
&=_{C_n} \alpha\alpha\delta \beta && \text{(by \eqref{eq:chineserela} applied to $\alpha\beta\delta$)} \\
&=_{C_n} \alpha\delta \alpha \beta && \text{(by \eqref{eq:chineserela} applied to $\alpha\alpha\delta$)} \\
&= \gamma \delta \alpha \beta.
\end{align*}
  Therefore $\csw_D(d_{\alpha\beta}d_{\gamma\delta}) = d_{\gamma
    \delta}d_{\alpha \beta}$. Since $\alpha =\gamma$ and $\beta >
  \delta$, we have $d_{\alpha\beta}\succ d_{\gamma \delta}$ and so
  $d_{\alpha\beta}d_{\gamma\delta}\succ_\lex d_{\gamma
    \delta}d_{\alpha \beta}$.

\item $\alpha >\gamma$. Now consider two sub-sub-cases separately:
\begin{enumerate}
\item $\beta \geq \gamma$. Then $\alpha >\beta \geq \gamma >\delta$
  and so
\begin{align*}
\alpha\beta\gamma\delta &=_{C_n} \alpha\gamma\beta\delta && \text{(by \eqref{eq:chineserela} applied to $\alpha\beta\gamma$)} \\
&=_{C_n} \beta\alpha\gamma\delta && \text{(by \eqref{eq:chineserelb} applied to $\alpha\gamma\beta$)} \\
&=_{C_n} \beta\alpha\delta\gamma && \text{(by \eqref{eq:chineserela} applied to $\alpha\gamma\delta$)} \\
&=_{C_n} \beta\gamma\alpha\delta. && \text{(by \eqref{eq:chineserelb} applied to $\alpha\delta\gamma$)}
\end{align*}
Consider two sub-sub-sub-cases separately:
\begin{enumerate}
\item $\beta=\gamma$. Then $\csw_D(d_{\alpha\beta}d_{\gamma\delta}) =
  d_{\beta}^2d_{\alpha\delta}$. Since $\alpha >\beta$, we have
  $d_{\alpha\beta} \succ d_\beta$ and so
  $d_{\alpha\beta}d_{\gamma\delta}\succ_\lex d_{\beta}^2d_{\alpha\delta}$.
\item $\beta >\gamma$. Then $\csw_D(d_{\alpha\beta}d_{\gamma\delta}) =
  d_{\beta \gamma}d_{\alpha\delta}$. Since $\alpha >\beta$, we have
  $d_{\alpha\beta} \succ d_{\beta\gamma}$ and so
  $d_{\alpha\beta}d_{\gamma\delta}\succ_\lex d_{\beta
    \gamma}d_{\alpha\delta}$.
\end{enumerate}

\item $\gamma >\beta$. Then $\alpha >\gamma >\beta,\delta$ and so
  $\alpha\beta\gamma\delta =_{C_n} \gamma \alpha \beta \delta$ by
  \eqref{eq:chineserelb} applied to $\alpha\beta\gamma$.  Depending on
  $\beta$ and $\delta$ we get two sub-sub-sub-cases:
\begin{enumerate}
\item $\beta \geq \delta$. Then
\begin{align*}
\gamma \alpha\beta\delta &=_{C_n}
   \gamma \alpha\delta\beta && \text{(by \eqref{eq:chineserela} applied to $\alpha\beta\delta$)} \\
&=_{C_n}   \gamma \beta \alpha \delta. && \text{(by \eqref{eq:chineserelb} applied to $\alpha\delta\beta$)}
\end{align*}
Therefore
  $\csw_D(d_{\alpha\beta}d_{\gamma\delta}) = d_{\gamma\beta}d_{\alpha
    \delta}$. Since $\alpha >\gamma$, we have $d_{\alpha\beta} \succ_\lex
  d_{\gamma\beta}$ and so $d_{\alpha\beta}d_{\gamma\delta} \succ_\lex
  d_{\gamma\beta}d_{\alpha \delta}$.
\item $\delta >\beta$. Then $\gamma\alpha\beta\delta =_{C_n} \gamma
  \delta \alpha \beta$ by \eqref{eq:chineserelb} applied to
  $\alpha\beta\delta$. Therefore
  $\csw(d_{\alpha\beta}d_{\gamma\delta}) = d_{\gamma \delta}
  d_{\alpha\beta}$. Since $\alpha >\gamma$, we have $d_{\alpha\beta}
  \succ_\lex d_{\gamma \delta}$ and so
  $d_{\alpha\beta}d_{\gamma\delta} \succ_\lex d_{\gamma \delta}
  d_{\alpha\beta}$.
\end{enumerate}
\end{enumerate}
\end{enumerate}
\end{enumerate}

For any $w\in D^*$, the word in $A^*$ that can be read from the
subscripts represents the same element of $C_n$ as $w$. Since the
original presentation \eqref{eq:chineserela}--\eqref{eq:chineserelb}
is homogeneous, if $w,w'\in D^*$ are such that $w =_{C_n} w'$, the
words in $A^*$ read from the subscripts have the same length. So there
are only finitely many possible words in $D^*$ representing each
element of $C_n$. In particular, the $(A,\rel{T})$ is
globally finite.

Notice that there is no word $w$ such that $w \succ_\lex w$. Hence
$(A,\rel{T})$ is acyclic. Since $(A,\rel{T})$ is globally finite and
acyclic, it is noetherian.
\end{proof}

Combining \fullref{Lemmata}{lem:irredchinese} and
\ref{lem:chinesenoetherian} and the fact that $D$-Chinese staircase
words form a language of unique normal forms for $C_n$, we obtain the
desired result:

\begin{theorem}
$(A,\rel{T})$ is a finite complete rewriting system presenting $C_n$.
\end{theorem}

\subsection{Algorithm for right-multiplication}

Cassaigne et al.\ \cite[\S~2.2]{cassaigne_chinese} give the following
algorithm that takes a symbol $\gamma \in A$ and the staircase
$\sigma$ corresponding to a Chinese staircase word $w$ and computes
the staircase corresponding to the unique staircase word equal to
$w\gamma$. By starting with the empty staircase diagram (where all the
entries are $0$, corresponding to the empty staircase word
$\emptyword$) and iteratively applying this algorithm, one can recover
the staircase word equal in $C_n$ to an arbitrary word. For the
purposes of describing this algorithm, the top $k$ rows of a staircase
diagram form a \defterm{$k$-staircase diagram}. A $k$-staircase
diagram corresponds to the \defterm{$k$-staircase word}
$\ell^{(1)}\dots \ell^{(k)}$, which is a prefix (not necessarily proper) of
a staircase word \eqref{eq:chinesenf}. The algorithm will recursively
work with a $k$-staircase diagram and a symbol $\gamma \leq k$; to
start the computation, we simply begin with $k = n$.

\begin{algorithm}
\label{alg:chineseright}
~\par\nobreak
\textit{Input:} A $k$-staircase diagram $\sigma$ corresponding to a
$k$-staircase word $\ell^{(1)}\cdots \ell^{(k)}$ and a symbol $\gamma \leq k$.

\textit{Output:} A $k$-staircase diagram $\sigma \cdot \gamma$
corresponding to the $k$-staircase word equal to $\ell^{(1)}\cdots
\ell^{(k)}\gamma$ in $C_n$.

\textit{Method:} Write $\sigma = (\sigma',R_1)$, where $R_1$ is the
bottom ($k$-th) row of the diagram and $\sigma'$ the remaining
$(k-1)$-staircase diagram.

\begin{enumerate}

\item If $\gamma = k$, then $\sigma \cdot \gamma = (\sigma',R_1')$, where
  $R_1'$ is obtained from $R_1$ adding $1$ to $\sigma_{k}$.

\item If $\gamma < k$, let $\beta$ be maximal such that the entry in
  column $\beta$ of $R_1$ is non-zero. If no such $\beta$ exists, then
  set $\beta = \gamma$.
\begin{enumerate}

\item If $\gamma \geq \beta$, then $\sigma \cdot \gamma = (\sigma' \cdot \gamma,R_1)$.

\item If $\gamma < \beta$ and $\beta < k$, then $\sigma \cdot \gamma =
  (\sigma'\cdot \beta,R_1')$, where $R_1'$ is obtained from $R_1$ by
  subtracting $1$ from $\sigma_{k\beta}$ and adding $1$ to
  $\sigma_{k\gamma}$.

\item If $\gamma < \beta$ and $\beta = k$, then $\sigma \cdot \gamma =
  (\sigma',R_1')$ where $R_1'$ is obtained from $R_1$ by subtracting
  $1$ from $\sigma_{k}$ and adding $1$ to $\sigma_{k\gamma}$.

\end{enumerate}
\end{enumerate}
\end{algorithm}

\subsection{Algorithm for left-multiplication}
\label{subsec:chineseleftmult}

When we construct the biautomatic structure for the Chinese monoid and
prove that the left-multiplication relation is recognized by a
finite transducer, we will use the following left-handed analogue of
\fullref{Algorithm}{alg:chineseright}. We believe this algorithm is
new and potentially of independent interest.

\begin{algorithm}
\label{alg:chineseleft}
~\par\nobreak \textit{Input:} A Chinese staircase $\sigma$
corresponding to a Chinese staircase word $w$, and a symbol $\gamma
\in A$.

\textit{Output:} A Chinese staircase $\gamma\cdot\sigma$ corresponding
to the unique Chinese staircase equal in $C_n$ to $\gamma w$.

\textit{Method:} Store a symbol $\beta$ from $\{1,\ldots,n\} \cup
\{\bot\}$, initially set to $\bot$. There are two stages. In the first
stage, iterate the following for $\rho = 1,\ldots,\gamma-1$:
\begin{enumerate}

\item If every entry in row $\rho$ is empty, do nothing for this row.

\item Otherwise, let $\eta$ be the index of the rightmost non-zero
  entry (that is, $\eta$ is minimal with $\sigma_{\rho\eta} > 0$). Then:
\begin{enumerate}

\item If $\beta = \bot$:

\begin{enumerate}

\item If $\eta < \rho$, decrement $\sigma_{\rho\eta}$ by $1$,
  increment $\sigma_{\rho}$ by $1$, and set $\beta = \eta$.

\item If $\eta = \rho$, decrement $\sigma_\rho$ by $1$ and set
  $\beta = \eta$.

\end{enumerate}

\item Otherwise, when $\beta \neq \bot$:

\begin{enumerate}

\item If $\eta < \beta$, decrement $\sigma_{\rho\eta}$ by $1$,
  increment $\sigma_{\rho\beta}$ by $1$, and set $\beta = \eta$.

\item Otherwise, when $\eta \geq \beta$, do nothing.

\end{enumerate}
\end{enumerate}
\end{enumerate}
In the second stage, for $\rho = \gamma$,
\begin{enumerate}

\item If $\beta = \bot$, increment $\sigma_\rho$ by $1$.

\item Otherwise, when $\beta \neq \bot$, increment $\sigma_{\rho\beta}$ by $1$.

\end{enumerate}
Finally, output the current staircase.
\end{algorithm}

\begin{proposition}
\fullref{Algorithm}{alg:chineseleft} always halts with the correct output.
\end{proposition}

\begin{proof}
First of all, notice that by setting $x = y$ in
\eqref{eq:chineserelb}, we obtain
\begin{equation}
x(zx) =_{C_n} (zx)x; \label{eq:chinesezxxcomm}
\end{equation}
and by setting $y=z$ in \eqref{eq:chineserela}, we obtain
\begin{equation}
z(zx) =_{C_n} (zx)z, \label{eq:chinesezzxcomm}
\end{equation}
that is, $zx$ commutes with $x$ and with $z$ for $x \leq
z$. Furthermore, for $w \leq x \leq y \leq z$, by applying
\eqref{eq:chineserelb} twice we obtain
\begin{equation}
\label{eq:chinesezwyxcomm}
(zw)(yx) = zwyx =_{C_n} yzwx =_{C_n} yxzw = (yx)(zw)
\end{equation}
that is, $zw$ and $yx$ commute for $w \leq x \leq y \leq z$.

We consider a staircase word \eqref{eq:chinesenf} and a symbol
$\gamma$ and the process of turning $\gamma \ell^{(1)}\cdots \ell^{(n)}$
into a staircase word $m^{(1)}\cdots m^{(n)}$ (where the $m^{(i)}$ are
the row words) using the relations
(\ref{eq:chineserela}--\ref{eq:chineserelb}). Let $i$ be minimal such
that $\ell^{(i)}$ is non-empty.

Depending on whether $i \geq \gamma$ or $i < \gamma$, we distinguish
two cases. Suppose first that $i \geq \gamma$. Then for $\rho =
1,\ldots,\gamma-1$, the iterative step in the algorithm always
proceeds via case~1, and these rows of the output staircase are the
same as those of the input staircase.  This is the correct output
since
\[
\gamma \ell^{(1)}\cdots \ell^{(n)} =_{C_n} \ell^{(1)}\cdots \ell^{(\gamma-1)}\gamma
\ell^{(\gamma)} \cdots \ell^{(n)}
\]
When the algorithm reaches the step for $\rho = \gamma$, it still has
$\beta = \bot$, and so it increments $\sigma_\gamma$ by $1$ and
halts. This produces a new word
\[
m^{(\gamma)} = (\gamma 1)^{\sigma_{\gamma 1}}\cdots(\gamma(\gamma-1))^{\sigma_{\gamma(\gamma-1)}}\gamma^{\sigma_\gamma+1}
\]
that is equal in $C_n$ to $\gamma \ell^{(\gamma)}$, by repeated
application of \eqref{eq:chinesezzxcomm}. Now, the resulting staircase
word $m^{(1)}\cdots m^{(n)}$ is equal in $C_n$ to $\gamma
\ell^{(1)}\cdots \ell^{(n)}$, since $m^{(\gamma)} =_{C_n} \gamma
\ell^{(\gamma)}$ and $m^{(j)} = \ell^{(j)}$ for $j \neq \gamma$. Hence in
this case the algorithm produces the correct output.

Suppose now that $i < \gamma$. Then the iterative step in the algorithm
proceeds via case~1 for $\rho = 1,\ldots,i-1$, and these rows of the
output staircase are the same as those of the input staircase. That
is, $m^{(j)} = \ell^{(j)}$ for $j = 1,\ldots,i-1$. This is the correct
output since
\[
\gamma \ell^{(1)}\cdots \ell^{(n)} =_{C_n} \ell^{(1)}\cdots \ell^{(i-1)}\gamma \ell^{(i)} \cdots \ell^{(n)}.
\]
For $\rho = i$ the algorithm proceeds via case~2(a). If $\eta < \rho$, then
\begin{align*}
&\gamma \ell^{(i)} \\
={}& \gamma (i \eta)^{\sigma_{i \eta}}\cdots(i(i-1))^{\sigma_{i(i-1)}}i^{\sigma_i} \\
={}& \gamma i\eta(i \eta)^{\sigma_{i \eta}-1}\cdots(i(i-1))^{\sigma_{i(i-1)}}i^{\sigma_i} && \text{(since $\sigma_{i\eta} \neq 0$)} \\
=_{C_n}{}& \gamma \eta i(i \eta)^{\sigma_{i \eta}-1}\cdots(i(i-1))^{\sigma_{i(i-1)}}i^{\sigma_i} &&\text{(by \eqref{eq:chineserela})} \\
=_{C_n}{}& \gamma \eta (i \eta)^{\sigma_{i \eta}-1}\cdots(i(i-1))^{\sigma_{i(i-1)}}i^{\sigma_i+1} &&\text{(by \eqref{eq:chinesezzxcomm})} \\
=_{C_n}{}& (i \eta)^{\sigma_{i \eta}-1}\cdots(i(i-1))^{\sigma_{i(i-1)}}i^{\sigma_i+1}\gamma\eta &&\text{(by \eqref{eq:chinesezwyxcomm}, and \eqref{eq:chineserelb} when $\sigma_i = 0$)} \\
={}& m^{(i)}\gamma\eta
\end{align*}
and the algorithm, proceeding via case~2(a)i, outputs the $i$-th row
corresponding to $m^{(i)}$ and now has $\beta = \eta$. (This
parameter $\beta$ works as a letter that is `dragged' by $\gamma$ as
it moves rightward through the word to its proper place.) If, on the
other hand, $\eta = i$, then
\begin{align*}
&\gamma \ell^{(i)} \\
={}& \gamma i^{\sigma_i} \\
={}& \gamma ii^{\sigma_i-1} && \text{(since $\sigma_i \neq 0$)} \\
=_{C_n}{}& i^{\sigma_i-1}\gamma i, &&\text{(by \eqref{eq:chinesezwyxcomm}, and \eqref{eq:chineserelb} when $\sigma_i = 2$)}
\end{align*}
and the algorithm, proceeding via case~2(a)ii, outputs the $i$-th row
corresponding to $m^{(i)} = i^{\sigma_i-1}$ and now has $\beta =
i$. Notice that in either case we have $\gamma \ell^{(i)} =_{C_n}
m^{(i)}\gamma\beta$ and $\beta \leq \rho$, and so in the next
iteration $\beta$ will be strictly less than the new value of
$\rho$. We shall see by induction that $\beta$ is always less than
$\rho$ when a new iteration begins.

For $\rho = i+1,\ldots,\gamma - 1$, the algorithm produces empty rows
whenever $\ell^{(\rho)}$ is empty and does not change $\beta$. So suppose
$\ell^{(\rho)}$ is not empty. Then since $\beta \neq \bot$, the algorithm
proceeds via case~2(b). We know that $\beta < \rho$. If $\eta <
\beta$, then
\begin{align*}
&\gamma\beta \ell^{(\rho)} \\
={}& \gamma\beta (\rho \eta)^{\sigma_{\rho \eta}}\cdots(\rho(\rho-1))^{\sigma_{\rho(\rho-1)}}\rho^{\sigma_\rho} \\
={}& \gamma\beta \rho\eta(\rho \eta)^{\sigma_{\rho \eta}-1}\cdots(\rho(\rho-1))^{\sigma_{\rho(\rho-1)}}\rho^{\sigma_\rho} && \text{(since $\sigma_{\rho\eta-1} \neq 0$)} \\
=_{C_n}{}& \rho\gamma\beta \eta(\rho \eta)^{\sigma_{\rho \eta}-1}\cdots(\rho(\rho-1))^{\sigma_{\rho(\rho-1)}}\rho^{\sigma_\rho}  &&\text{(by \eqref{eq:chineserelb})} \\
=_{C_n}{}& \rho\gamma\eta\beta(\rho \eta)^{\sigma_{\rho \eta}-1}\cdots(\rho(\rho-1))^{\sigma_{\rho(\rho-1)}}\rho^{\sigma_\rho}  &&\text{(by \eqref{eq:chineserela})} \\
=_{C_n}{}& \gamma\eta\rho\beta(\rho \eta)^{\sigma_{\rho \eta}-1}\cdots(\rho(\rho-1))^{\sigma_{\rho(\rho-1)}}\rho^{\sigma_\rho}  &&\text{(by \eqref{eq:chineserelb})} \\
=_{C_n}{}& \rho\beta(\rho \eta)^{\sigma_{\rho \eta}-1}\cdots(\rho(\rho-1))^{\sigma_{\rho(\rho-1)}}\rho^{\sigma_\rho}\gamma\eta  &&\text{(by \eqref{eq:chinesezwyxcomm}, and \eqref{eq:chineserelb} for $\sigma_\rho=1$)} \\
=_{C_n}{}& (\rho \eta)^{\sigma_{\rho \eta}-1}\cdots(\rho \beta)^{\sigma_{\rho \beta}+1}\cdots(\rho(\rho-1))^{\sigma_{\rho(\rho-1)}}\rho^{\sigma_\rho}\gamma\eta  &&\text{(by \eqref{eq:chinesezwyxcomm}),} \\
={}& m^{(\rho)}\gamma\eta,
\end{align*}
and the algorithm, proceeding via case~2(b)i, outputs the $\rho$-th row corresponding to $m^{(\rho)}$ and now has $\beta
= \eta$. If, on the other hand, $\eta \geq \beta$, then
\begin{align*}
&\gamma\beta \ell^{(\rho)} \\
={}& \gamma\beta (\rho\eta)^{\sigma_{\rho\eta}}\cdots(\rho(\rho-1))^{\sigma_{\rho(\rho-1)}}\rho^{\sigma_\rho} \\
=_{C_n}{}& (\rho\eta)^{\sigma_{\rho\eta}}\cdots(\rho(\rho-1))^{\sigma_{\rho(\rho-1)}}\rho^{\sigma_\rho}\gamma\beta  &&\text{(by \eqref{eq:chinesezwyxcomm}, and \eqref{eq:chineserelb} when $\sigma_\rho=1$)}\\
={}& m^{(\rho)}\gamma\beta.
\end{align*}
and the algorithm, proceeding via case~2(b)ii, outputs the same $\rho$-th
row (corresponding to $m^{(\rho)} = \ell^{(\rho)}$) and has $\beta$
unchanged. Notice that in either case $\gamma\beta \ell^{(\rho)} =_{C_n}
m^{(\rho)}\gamma\hat\beta$, where $\hat\beta$ is the new value of
$\beta$ obtained by executing the algorithm. Also, $\beta \leq \rho$,
and so in the next iteration $\beta$ will be strictly less than the
new value of $\rho$.

Finally, for $\rho = \gamma$,
\begin{align*}
&\gamma\beta \ell^{(\rho)} \\
={}& \rho\beta (\rho \eta)^{\sigma_{\rho \eta}}\cdots(\rho(\rho-1))^{\sigma_{\rho(\rho-1)}}\rho^{\sigma_\rho} &&\text{(since $\gamma=\rho$)}\\
=_{C_n}{}& (\rho \eta)^{\sigma_{\rho \eta}}\cdots(\rho \beta)^{\sigma_{\rho \beta}+1}\cdots(\rho(\rho-1))^{\sigma_{\rho(\rho-1)}}\rho^{\sigma_\rho}  &&\text{(by \eqref{eq:chinesezwyxcomm}),} \\
={}& m^{(\rho)}
\end{align*}
and so the algorithm, proceeding via case~2, outputs the
$\rho$-th row corresponding to $m^{(\rho)}$ and halts.

For $j > \rho$, we have $\ell^{(j)} = m^{(j)}$ and so, at the end of
the algorithm, we have the staircase diagram corresponding to the
Chinese staircase word $m^{(1)}\cdots m^{(n)}$.
\end{proof}

\subsection{Biautomaticity}
\label{subsec:chinesebiauto}

The essential idea in constructing our biautomatic structure is to
have a language of normal forms related to \eqref{eq:chinesenf} and to
show that \fullref{Algorithms}{alg:chineseright} and
\ref{alg:chineseleft} can be performed by a transducer. We will
initially use the generating set $D$ introduced in
\fullref{\S}{subsec:chinesefcrs}, but we will later switch to the
standard generating set $A$.

Let
\[
K^{(k)} = d_{k 1}^*d_{k 2}^*\cdots d_{k(k-1)}^*d_k^*
\]
and
\[
K = K^{(1)}K^{(2)}\cdots K^{(n-1)}K^{(n)};
\]
notice that $K$ is regular.

\subsubsection{Right-multiplication by transducer}

Let us show that the relation $K_{d_{\gamma}} = \{(u,v) : u,v \in K,
ud_\gamma =_{C_n} v\}$ is recognized by a finite transducer for any $\gamma
\in \{1,\ldots,n\}$. We only need to consider right multiplication by
$d_\gamma$, where $\gamma \in A$, because we will later switch back to
the generating set $A$.

We imagine a transducer reading a pair of words in $K$ \emph{from
  right to left} with the aim of checking whether this pair is in
$K_{d_{\gamma}}$. It is easiest to describe the transducer as reading
symbols from the top tape and outputting symbols on the bottom
tape. As it reads a symbol from the top tape, the transducer knows
which language $K^{(k)}$ (corresponding to the $k$-th row) the current
symbol comes from, for symbols $d_{k\beta}$ and $d_k$ occur only in
subwords from $K^{(k)}$. The transducer stores a symbol $\alpha$ from
$\{1,\ldots,n\} \cup \{\infty\}$ that records the symbol to be
inserted into the $(k-1)$-staircase word from $K^{(1)}\cdots
K^{(k-1)}$. Initially $\alpha = \gamma$. The value $\alpha = \infty$
indicates that no further insertion is necessary. In this case the
transducer simply reads symbols from the top tape and outputs the
same symbols on the bottom tape until the end of input.

When $\alpha \neq \infty$, the transducer functions as follows. As
remarked above, the transducer knows when it starts reading a subword
from $K^{(k)}$. If $\alpha = k$, the transducer outputs an extra
symbol $d_k$ (corresponding to case~1 in
\fullref{Algorithm}{alg:chineseright}) and sets $\alpha = \infty$. If
$\alpha < k$, the transducer needs to calculate $\beta$ in accordance
with case~2 of \fullref{Algorithm}{alg:chineseright}. But $\beta$ can
be determined from the rightmost symbol of the subword from $K^{(k)}$,
which is the first symbol of this subword it encounters, since it is
reading right-to-left: this symbol is $d_{k\beta}$ or $d_k$, in which
case $\beta = k$. [If there is no symbol from $K^{(k)}$ on the top
  tape, then the next symbol it reads is from $K^{(j)}$ for some $j <
  k$. No special action is required here, for the automaton now
  proceeds to insert $\alpha$ into the subword from $K^{(1)}\cdots
  K^{(j)}$, which is in accordance with the recursion in case 2(a) of
  \fullref{Algorithm}{alg:chineseright}.]

If $\alpha \geq \beta$, the transducer reads each symbol of the
subword from $K^{(k)}$ and outputs it on the bottom tape, all the while
keeping the same value of $\alpha$. It then arrives at the rightmost
end of the next subword with the correct value of $\alpha$ (this
corresponds to the recursion in case 2(a) of
\fullref{Algorithm}{alg:chineseright}).

If $\alpha < \beta$ and $\beta < k$, then the transducer reads each
symbol of the subword from $K^{(k)}$ and outputs it on the bottom
tape, \emph{with the following exceptions}:
\begin{itemize}

\item It does not output anything on reading the first symbol
  $d_{k\beta}$ (corresponding to decrementing $\sigma_{k\beta}$ by
  $1$).

\item It outputs an extra symbol $d_{k\alpha}$ immediately after
  reaching the leftmost end of the (possibly empty) subword of symbols
  $d_{k\alpha}$ (it can non-deterministically predict when it has
  reached the end of this subword; this corresponds to incrementing
  $\sigma_{k\alpha}$ by $1$).

\end{itemize}
On reaching the end of the subword from $K^{(k)}$, it sets $\alpha =
\beta$, corresponding to the recursion in case~2(b) of
\fullref{Algorithm}{alg:chineseright}).

If $\alpha < \beta$ and $\beta = k$, then the transducer reads each
symbol of the subword from $K^{(k)}$ and outputs it on the bottom
tape, \emph{with the following exceptions}:
\begin{itemize}

\item It does not output anything on reading the first symbol
  $d_{k}$ (corresponding to decrementing $\sigma_{k}$ by
  $1$).

\item It outputs an extra symbol $d_{k\alpha}$ immediately after
  reaching the leftmost end of the (possibly empty) subword of symbols
  $d_{k\alpha}$ (it can non-deterministically predict when it has
  reached the end of this subword; this corresponds to incrementing
  $\sigma_{k\alpha}$ by $1$).

\end{itemize}
On reaching the end of the subword from $K^{(k)}$, it sets $\alpha =
\infty$, corresponding to case~2(c) of
\fullref{Algorithm}{alg:chineseright}).

This shows that $K_{d_\gamma}$ is recognized by a transducer;
$K_{d_\gamma}$ is therefore a rational relation.

\subsubsection{Left-multiplication by transducer}

Let us show that the relation ${}_{d_{\gamma}}K = \{(u,v) : u,v \in K,
d_\gamma u =_{C_n} v\}$ is recognized by a finite transducer for any
$\gamma \in \{1,\ldots,n\}$.

We imagine a transducer reading a pair of words in $K$ from \emph{left to
right} and checking whether this pair is in ${}_{d_{\gamma}}K$. As
before, we will describe the transducer as reading symbols from the
top tape and outputting symbols on the bottom tape. As it reads a
symbol from the top tape, the transducer knows which $K^{(k)}$ the
current symbol comes from, for symbols $d_{k\alpha}$ and $d_k$ occur
only in subwords from $K^{(k)}$. The transducer also
non-deterministically looks ahead one symbol on the top tape, so it
knows when it has reached the end of a substring (possibly empty) of
symbols $d_{k\alpha}$ or $d_k$.

In its state, the transducer stores the symbol $\beta \in
\{1,\ldots,n\} \cup \{\bot,\infty\}$. Initially $\beta = \bot$. The
value $\beta = \infty$ indicates that no further changes are necessary
and the transducer simply reads symbols from the top tape and outputs
the same symbols on the bottom tape. When $\beta \neq \infty$, the
transducer functions as follows.

Suppose it is reading some subword from $K^{(\rho)}$ from the top
tape, where $\rho < \gamma$. If this subword is empty, then the
corresponding output subword is empty, as per case~1 of the iterative
part of \fullref{Algorithm}{alg:chineseleft}. When this subword is
non-empty, and the algorithm proceeds via case~2(a) or 2(b), the
transducer can determine $\eta$ from the first symbol
$d_{\rho\eta}$ or $d_\rho$ of the subword. It reads each symbol of
this subword and outputs the same symbol \emph{with the following
  exceptions}:
\begin{itemize}

\item If $\beta = \bot$ and $\eta < \rho$ (case~2(a)i), it outputs
  $\emptyword$ on reading this symbol $d_{\rho\eta}$ (corresponding
  to decrementing $\sigma_{\rho\eta}$ by $1$), outputs an extra
  symbol $d_\rho$ immediately after reading the last symbol of this
  subword (corresponding to incrementing $\sigma_{\rho}$ by $1$), and
  sets $\beta = \eta$.

\item If $\beta = \bot$ and $\eta = \rho$ (case~2(a)ii), it outputs
  $\emptyword$ on reading the first symbol $d_\rho$
  (corresponding to decrementing $\sigma_{\rho}$ by $1$), and sets
  $\beta = \eta$.

\item If $\beta \neq \bot$ and $\eta < \beta$ (case~2(b)i), it
  outputs $\emptyword$ on reading this symbol $d_{\rho\eta}$
  (corresponding to decrementing $\sigma_{\rho\eta}$ by $1$),
  outputs an extra symbol $d_{\rho\beta}$ immediately after reading
  the last symbol $d_{\rho\beta}$ (corresponding to incrementing
  $\sigma_{\rho\beta}$ by $1$), and sets $\beta = \eta$.

\end{itemize}
[Notice that in case~2(b)ii, the transducer simply reads symbols and
  outputs them: there are no exceptions in this case.]

Finally, on reading the subword from $K^{(\gamma)}$, again it reads
symbols and outputs them, except that if $\beta = \bot$ it outputs an
extra symbol $d_\gamma$ at the end of the (possibly empty) string of
symbols $d_\gamma$, and if $\beta \neq \bot$, it outputs an extra
symbol $d_{\gamma\beta}$ at the end of the (possibly empty) string of
symbols $d_{\gamma\beta}$. On doing this, it sets $\beta = \infty$ and
simply reads to the end of the top tape and outputs the same symbols
on the bottom tape.

This shows that ${}_{d_\gamma}K$ is recognized by a transducer;
${}_{d_\gamma}K$ is therefore a rational relation.

\subsubsection{Deducing biautomaticity}

Let $\rel{Q} \subseteq D^* \times A^*$ be the rational relation
\[
\bigl\{(d_{\alpha\beta},\alpha\beta),(d_\alpha,\alpha) : \alpha,\beta \in A, \alpha > \beta\bigr\}^*.
\]
Let
\[
L = K \circ \rel{Q} = \bigl\{u \in A^* : (\exists u' \in K)((u',u) \in \rel{Q})\bigr\};
\]
then $L$ is a regular language over $A$ that maps onto $C_n$. [In
  fact, $L$ is the set of Chinese staircase words, but we do not need
  to know this.] Then for any $\gamma \in A$,
\begin{align*}
(u,v) \in L_\gamma &\iff u \in L \land v \in L \land u\gamma =_{C_n} v\\
&\iff (\exists u',v' \in K)((u',u) \in \rel{Q} \land (v',v) \in \rel{Q} \land u'd_\gamma =_{C_n} v')\\
&\iff (\exists u',v' \in K)((u',u) \in \rel{Q} \land (v',v) \in \rel{Q} \land (u',v') \in K_{d_\gamma})\\
&\iff (u,v) \in \rel{Q}^{-1} \circ K_{d_\gamma} \circ \rel{Q}.
\end{align*}
Therefore $L_\gamma$ is a rational relation. Now, if $(u,v) \in
L_\gamma$, then $|v| = |u|+1$ since the defining relations
(\ref{eq:chineserela}--\ref{eq:chineserelb}) preserve lengths of words
over $A$. By \fullref{Proposition}{prop:rationalbounded},
$L_\gamma\rpad$ and $L_\gamma\lpad$ are regular.

Similarly, from the fact that ${}_{c_\gamma}K$ is a rational relation,
we deduce that ${}_\gamma L = \rel{Q}^{-1} \circ {}_{d_\gamma}K \circ
\rel{Q}$ is a rational relation and thus, by
\fullref{Proposition}{prop:rationalbounded}, that ${}_\gamma L\rpad$
and ${}_\gamma L\lpad$ are regular.

This proves the result:

\begin{theorem}
\label{thm:chinesebiauto}
$(A,L)$ is a biautomatic structure for the Chinese monoid $C_n$.
\end{theorem}

\section{Hypoplactic monoid}

\subsection{Quasi-ribbon tableaux}

Let $n \in \nset$. Let $A$ be the ordered alphabet $\{1 < 2 < \ldots <
n\}$. Let $\rel{R}$ be the set of defining relations for the Plactic monoid; that is,
\begin{align*}
\rel{R} ={}& \{(acb,cab) : a \leq b < c\} \cup \{(bac, bca) : a < b \leq c\}.
\end{align*}
and let
\[
\rel{S} = \{(cadb,acbd), (bdac,dbca) : a \leq b < c \leq d\}.
\]
The \defterm{hypoplactic monoid of rank $n$}, denoted $H_n$, is
presented by $\pres{A}{\rel{R} \cup \rel{S}}$.

A \defterm{column} is a strictly decreasing word in $A^*$ (that is, a
word $\alpha = \alpha_1\cdots\alpha_k$, where $\alpha_i \in A$, such
that $\alpha_i > \alpha_{i+1}$ for all $i = 1,\ldots,k-1$). Any word
$\alpha \in A^*$ has a decomposition as a product of columns of
maximal length $\alpha = \alpha^{(1)}\cdots\alpha^{(k)}$. Such a word
$\alpha$ is a \defterm{quasi-ribbon word} if the last (smallest) symbol
of $\alpha^{(i+1)}$ is greater than or equal to the first (greatest)
symbol of $\alpha^{(i)}$ for all $i = 1,\ldots,k-1$.

For example, $1\;1\;1\;\allowbreak 21\;2\;\allowbreak
543\;65\;\allowbreak 6\;7\;87$ is a quasi-ribbon word. (For clarity,
spaces indicate the decomposition into columns of maximum length.) Any
quasi-ribbon word has a planar representation as a
\defterm{quasi-ribbon tableau}, where the columns are written
vertically from bottom to top and arranged left to right so that the
last (uppermost) symbol in each column aligns with the first symbol of
the previous column. The quasi-ribbon tableau corresponding to
$1\;1\;1\;\allowbreak 21\;2\;\allowbreak 543\;65\;\allowbreak
6\;7\;87$ is shown in \fullref{Figure}{fig:qrt}; notice that each row
in the quasi-ribbon tableau is non-decreasing.

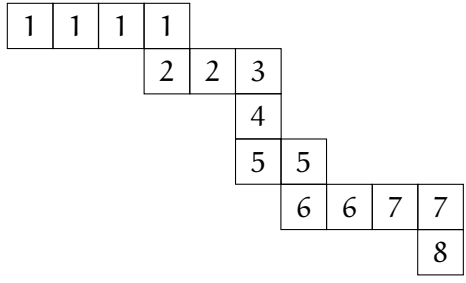
\begin{figure}[t]
\hbox to \hsize{\hss%
\begin{tikzpicture}
\matrix (quasiribbon) [
matrix of math nodes,
nodes={rectangle,draw,minimum width=6mm,minimum height=6mm},
row sep={between borders,-\pgflinewidth},
column sep={between borders,-\pgflinewidth},
]
{
  1 & 1 & 1 & 1 &   &   &   &   &   &   \\
    &   &   & 2 & 2 & 3 &   &   &   &   \\
    &   &   &   &   & 4 &   &   &   &   \\
    &   &   &   &   & 5 & 5 &   &   &   \\
    &   &   &   &   &   & 6 & 6 & 7 & 7 \\
    &   &   &   &   &   &   &   &   & 8 \\
};
\end{tikzpicture}\hss}
\caption{Example of a quasi-ribbon tableau.}
\label{fig:qrt}
\end{figure}

\begin{theorem}[{\cite[Theorem~4.17]{novelli_hypoplactic}}]
\label{thm:qrtcross}
The quasi-ribbon words form a cross-section of the hypoplactic monoid.
\end{theorem}

For any $w \in A^*$, let $Q(w)$ denote the unique quasi-ribbon word
such that $w =_{H_n} Q(w)$.

\begin{theorem}[{\cite[Theorem~5.12]{novelli_hypoplactic}}]
\label{thm:qrtmin}
For any $w \in A^*$, the quasi-ribbon word $Q(w)$ is the
lexicographically minimum word such that $w =_{H_n} Q(w)$. That is,
$Q(w) \leq_{\lex} w$ for all $w \in A^*$.
\end{theorem}

\begin{algorithm}[{\cite[Algorithm~4.4]{novelli_hypoplactic}}]
\label{alg:novelli}
~\par\nobreak
\textit{Input:} A quasi-ribbon word $w$ and a symbol $a$.

\textit{Output:} The quasi-ribbon word $Q(wa)$.

\textit{Method:}

Let $T$ be the quasi-ribbon tableau corresponding to $w$. If there is
no entry in $T$ that is less than or equal to $a$, output the word
corresponding to the tableau obtained by putting $a$ and gluing $T$ to
the bottom of $a$.

Otherwise, let $x$ be the right-most and bottom-most entry of $T$ that
is less than or equal to $x$. Put a new entry $a$ to the right of $x$
and glue the remaining part of $T$ (below and to the right of $x$)
onto the bottom of the new entry $a$. Output the quasi-ribbon word
corresponding to the new tableau.
\end{algorithm}

The symmetry of the presentation $\pres{A}{\rel{R}\cup\rel{S}}$ and
of the definition of quasi-ribbon words means that
\fullref{Algorithm}{alg:novelli} has the following symmetrical
version, describing how to left-multiply a quasi-ribbon word by a
generator.

\begin{algorithm}
\label{alg:novellileft}
~\par\nobreak
\textit{Input:} A quasi-ribbon word $w$ and a symbol $a$.

\textit{Output:} The quasi-ribbon word $Q(aw)$.

\textit{Method:}

Let $T$ be the quasi-ribbon tableau corresponding to $w$. If there is
no entry in $T$ that is greater than or equal to $a$, output the word
corresponding to the tableau obtained by putting $a$ and gluing $T$ to
the top of $a$.

Otherwise, let $x$ be the left-most and upper-most entry of $T$ that
is greater than or equal to $x$. Put a new entry $a$ to the left of
$x$ and glue the remaining part of $T$ (above and to the left of $x$)
onto the top of the new entry $a$. Output the quasi-ribbon word
corresponding to the new tableau.
\end{algorithm}

\subsection{Finite complete rewriting system}
\label{subsec:hypoplacticfcrs}

Let $\rel{T} = \{w \imreduces Q(w) : w \in A^* \land w \neq Q(w) \land |w| \leq \max\{2n,4\}\}$.

\begin{theorem}
$(A,\rel{T})$ is a finite complete rewriting system presenting $H_n$.
\end{theorem}

\begin{proof}
First of all notice that every rules in $\rel{T}$ holds in $H_n$ and
thus is a consequence of the relations in $\rel{R} \cup \rel{S}$. On the other
hand, every relation in $\rel{R} \cup \rel{S}$ is a consequence of the rules in
$\rel{T}$ (since $\rel{T}$ includes all rules $w \imreduces Q(w)$ for
all $|w| \leq 4$). Hence $\pres{A}{\rel{T}}$ presents $H_n$.

Next, notice that there are only finitely many rules in $\rel{T}$
since there are only finitely many possibilities for $w$ and $Q(w)$ is
uniquely determined.

Suppose $u \imreduces_\rel{T} v$. Then $|u| = |v|$ and $v <_{\lex} u$
by \fullref{Theorem}{thm:qrtmin} (since a rule $w \imreduces Q(w)$ is
applied where $w \neq Q(w)$ and $Q(w) \leq_{\lex} w$ for all $w \in
A^*$). Hence $(A,\rel{T})$ is acyclic. Furthermore, since $|w| =
|Q(w)|$, rewriting using $(A,\rel{T})$ preserves lengths, and so for
each $u \in A^*$, there are only finitely many words $v$ such that $u
\reduces_{\rel{T}} v$; thus $(A,\rel{T})$ is globally finite. Hence
$(A,\rel{T})$ is noetherian.

Let $u \in A^*$. Since $(A,\rel{T})$ is noetherian, applying rewriting
rules from $\rel{T}$ to $u$ will always yield an irreducible word
$v$.

With the aim of obtaining a contradiction, suppose that $v$ is not a
quasi-ribbon word. Let $v = v^{(1)}v^{(2)}\cdots v^{(k)}$ be the
decomposition of $v$ into columns of maximum length. Since $v$ is not
a quasi-ribbon word, for some $i$ the word $v^{(i)}v^{(i+1)}$ is not a
quasi-ribbon word. Since columns are strictly descending words, they
have maximum length $n$. So $v^{(i)}v^{(i+1)}$ has length at most
$2n$. Since $v^{(i)}v^{(i+1)}$ is not a quasi-ribbon word,
$v^{(i)}v^{(i+1)} \neq Q(v^{(i)}v^{(i+1)})$. Hence $v^{(i)}v^{(i+1)}
\imreduces Q(v^{(i)}v^{(i+1)})$ is a rule in $\rel{T}$. Since $v$
contains the left-hand side of this rule, it is not irreducible, which
is a contradiction. This proves that $v$ is a quasi-ribbon word.

Hence the set of irreducible words for $(A,\rel{T})$ are the
quasi-ribbon words, which form a cross-section of $H_n$ by
\fullref{Theorem}{thm:qrtcross}. Hence $(A,\rel{T})$ is confluent.
\end{proof}

\subsection{Biautomaticity}
\label{subsec:hypoplacticbiauto}

\subsubsection{Right- and left-multiplication by transducer}

Let
\[
C = \{c_\alpha : \text{$\alpha \in A^+$ is a column}\}.
\]
Each symbol $c_\alpha$ represents the element $\alpha$ of $H_n$. Since
this holds in particular when $|\alpha| = 1$, the original generating
set $A$ is essentially contained in $C$. Thus $C$ also generates $H_n$.

Define a relation $\preceq$ on columns by $\alpha \preceq\beta$ if the
first (greatest) symbol of $\alpha$ is less than or equal to the last
(smallest) symbol of $\beta$. That is, $\alpha \preceq \beta$ if
$\alpha$ can appear immediately to the left of $\beta$ in the
decomposition of a quasi-ribbon word into maximal columns.

Now let
\[
K = \bigl\{c_{\alpha^{(1)}}c_{\alpha^{(2)}}\cdots c_{\alpha^{(k)}} : k \in \nset \cup \{0\}, c_{\alpha^{(i)}} \in C, {\alpha^{(j)}} \preceq {\alpha^{(j+1)}} \text{ for all $j$}\bigr\}.
\]
Notice that $K$ is regular since an automaton checking whether
$c_{\alpha^{(1)}}c_{\alpha^{(2)}}\cdots c_{\alpha^{(k)}} \in C^*$ lies
in $K$ need only store the previous letter in its state in order to
check whether ${\alpha^{(j)}} \preceq {\alpha^{(j+1)}}$ for all $j$.

We will first of all prove that for any $\gamma \in A$ the relation
$K_{c_{\gamma}}$ is recognized by a finite transducer.

We imagine a transducer reading a pair of words
\[
(c_{\alpha^{(1)}}\cdots c_{\alpha^{(k)}},c_{\beta^{(1)}}\cdots c_{\beta^{(l)}}) \in K \times K
\]
\emph{from right to left}, with the aim of checking whether this pair
is in $K_{c_{\gamma}}$. It is easiest to describe the transducer as
reading symbols from the top tape and outputting symbols on the bottom
tape.  Essentially, the transducer will perform
\fullref{Algorithm}{alg:novelli} using the alphabet $C$ to store
columns quasi-ribbon words.

The transducer non-deterministically looks one symbol ahead (that is,
further left) on the top tape. In its state it stores either
$\gamma$ or $\infty$, with the latter indicating that $\gamma$ has
already been inserted into the quasi-ribbon. 

While the transducer is storing $\gamma$ in its state, it examines
each column $c_{\alpha^{(j)}}$ it reads and proceeds as follows:
\begin{itemize}
\item If $\gamma$ is less than every symbol of $\alpha^{(j)}$ and $j >
  1$ and $\gamma$ is less than the first symbol of $\alpha^{(j-1)}$,
  then the transducer outputs $c_{\alpha^{(j)}}$ and proceeds to read
  the next symbol $c_{\alpha^{(j-1)}}$. (Notice that the transducer
  knows whether $j > 1$ since it non-deterministically looks ahead one
  symbol.)
\item If $\gamma$ is greater than or equal to some symbol of
  $\alpha^{(j)}$, then the transducer outputs
  $c_{\beta'}c_{\beta\gamma}$, where $\alpha^{(j)} = \beta\beta'$ and
  the first letter of $\beta'$ is the first symbol of $\alpha^{(j)}$
  that is less than or equal to $\gamma$. It then stores $\infty$ in
  its state in place of $\gamma$.
\item If $\gamma$ is less than every symbol of $\alpha^{(j)}$ and
  either $j=1$ or $j > 1$ and $\gamma$ is greater than or equal to the
  first symbol of $\alpha^{(j-1)}$, then the transducer outputs
  $c_{\alpha^{(j)}\gamma}$.  It then stores $\infty$ in its state in
  place of $\gamma$.
\end{itemize}
When the transducer is storing $\infty$ in its state, it simply reads
every input symbol and outputs the same symbol until it reaches the
end of the input.

Finally, if the top tape is the empty word $\emptyword$, the
transducer simply outputs $\gamma$.

Since $K_{c_\gamma}$ is recognized by a finite transducer, it is a
rational relation.

Since \fullref{Algorithms}{alg:novelli} and \ref{alg:novellileft} are
symmetric, it is clear that ${}_{c_\gamma}K$ is also recognized by a
finite transducer (which effectively performs
\fullref{Algorithm}{alg:novellileft} using the alphabet $C$ to store
column quasi-ribbon words). So ${}_{c_\gamma}K$ is a rational
relation.

\subsubsection{Deducing biautomaticity}

Let $\rel{Q} \subseteq C^* \times A^*$ be the relation
\[
\bigl\{(c_{\alpha^{(1)}}c_{\alpha^{(2)}}\cdots c_{\alpha^{(k)}},\alpha^{(1)}\alpha^{(2)}\cdots\alpha^{(k)}) : k \in \nset\cup \{0\}, \text{each $\alpha^{(i)}$ is a column}\bigr\}.
\]
It is easy to see that $\rel{Q}$ is a rational relation. Let
\[
L = K \circ \rel{Q} = \bigl\{u \in A^* : (\exists u' \in K)\bigl((u',u) \in \rel{Q}\bigr)\bigr\}.
\]
Then $L$ is a regular language over $A$ that maps onto $M_n$,
since the set of regular languages is closed under applying rational
relations. (In fact, $L$ is the set of quasi-ribbon words, but this is
not important for proving biautomaticity.) Then for any $\gamma \in
A$,
\begin{align*}
(u,v) \in L_\gamma &\iff u \in L \land v \in L \land u\gamma =_{H_n} v\\
&\iff (\exists u',v' \in K)((u',u) \in \rel{Q} \land (v,v') \in \rel{Q} \land u'c_\gamma =_{H_n} v')\\
&\iff (\exists u',v' \in K)((u',u) \in \rel{Q} \land (v,v') \in \rel{Q} \land (u',v') \in K_{c_\gamma})\\
&\iff (u,v) \in \rel{Q}^{-1} \circ K_{c_\gamma} \circ \rel{Q}.
\end{align*}
Therefore, $L_\gamma$ is a rational relation. Now, if $(u,v) \in
L_\gamma$, then $|v| = |u|+1$ since $u\gamma =_{H_n} v$ and the
defining relations $\rel{R} \cup \rel{S}$ preserve lengths of words.
By \fullref{Proposition}{prop:rationalbounded}, $L_\gamma\rpad$ and
$L_\gamma\lpad$ are regular.

Similarly, from the fact that ${}_{c_\gamma}K$ is a rational relation,
we deduce that ${}_\gamma L = \rel{Q}^{-1} \circ {}_{d_\gamma}K \circ
\rel{Q}$ is a rational relation and thus, by
\fullref{Proposition}{prop:rationalbounded}, that ${}_\gamma L\rpad$
and ${}_\gamma L\lpad$ are regular.

Thus we have proved the desired result:

\begin{theorem}
\label{thm:hypoplacticbiauto}
$(A,L)$ is a biautomatic structure for the hypoplactic monoid $H_n$.
\end{theorem}

\section{Sylvester monoid}

Let $n \in \nset$. Let $A$ be the finite ordered alphabet $\{1 < 2 <
\ldots < n\}$. Let $\rel{R}$ be the (infinite) set of defining relations
\[
\{(cavb,acvb) : a \leq b < c, v \in A^*\}.
\]
Then the \defterm{sylvester monoid of rank $n$}, denoted $S_n$, is
presented by $\pres{A}{\rel{R}}$. This is simply a restriction to
finite rank of the sylvester monoid as defined by Hivert et
al.~\cite[Definition~8]{hivert_algebra}.

\subsection{Complete rewriting system}
\label{subsec:sylvestercrs}

The aim of this section is to prove that $(A,\rel{R})$ is a complete
rewriting system. The proof will depend on two results proved by
Hivert et~al.~(\fullref{Propositions}{prop:bstcrossec} and
\ref{prop:lrpbstlexmin} below), and we need to define some new
concepts first.

A \defterm{(right strict) binary search tree} is a labelled rooted
binary tree where the label of each node is greater than or equal to
the label of every node in its left subtree, and strictly less than
every node in its right subtree; see the example in
\fullref{Figure}{fig:exbst}.

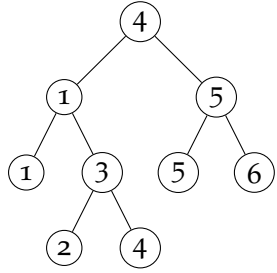
\begin{figure}[t]
\centering
\begin{tikzpicture}[every node/.style={circle,draw,inner sep=.7mm},level distance=10mm,level 1/.style={sibling distance=20mm},level 2/.style={sibling distance=10mm}]
  \node (root) {4}
    child { node (0) {1}
      child { node (00) {1} }
      child { node (01) {3} 
        child { node (010) {2} }
        child { node (011) {4} } } }
    child { node (1) {5}
      child { node (10) {5} }
      child { node (11) {6} } };
\end{tikzpicture}
\caption{Example of a binary search tree $T$. The root has label $4$,
  so every label in the left subtree of the root is less than or equal
  to $4$ (and indeed the label $4$ does occur) and every label in the
  right subtree of the root is strictly greater than $4$. Notice that
  (for example) $T = \bst(265415314)$, and that $LRP(T) =
  124315654$.}
\label{fig:exbst}
\end{figure}

Given a binary search tree $T$ and a symbol $a \in A$, one inserts $a$
into $T$ as follows: if $T$ is empty, create a node and label it
$a$. If $T$ is non-empty, examine the label $x$ of the root node; if
$a \leq x$, recursively insert $a$ into the left subtree of the root
node; otherwise recursively insert $a$ into the right subtree of the
root note. Denote the resulting tree $a \cdot T$. It is easy to see
that $a\cdot T$ is also a binary search tree.

Given any word $w \in A^*$, define its corresponding binary search
tree $\bst(w)$ as follows: start with the empty tree and iteratively
insert the symbols in $w$ from right to left; again, see the example
in \fullref{Figure}{fig:exbst}.

The left-to-right postfix reading $LRP(T)$ of a binary search tree $T$
is defined to be the word obtained as follows: recursively perform the
left-to-right postfix reading of the left subtree of the root of $T$,
then recursively perform the left-to-right postfix reading of the
right subtree of the root of $T$, then output the label of the root of
$T$; again, see the example in \fullref{Figure}{fig:exbst}. Note that
$\bst(LRP(T)) = T$.

\begin{proposition}[{\cite[Theorem~10]{hivert_algebra}}]
\label{prop:bstcrossec}
Let $w,w' \in A^*$. Then $w =_{S_n} w'$ if and only if $\bst(w) =
\bst(w')$.
\end{proposition}

\begin{proposition}[{\cite[Proposition~15]{hivert_algebra}}]
\label{prop:lrpbstlexmin}
Let $w \in A^*$. Then $LRP(\bst(w))$ is the lexicographically minimal
word representing the same element of $S_n$ as $w$.
\end{proposition}

\begin{proposition}
\label{prop:sylvesterfcrs}
$(A,\rel{R})$ is a complete rewriting system, and the irreducible
words are the lexicographically minimal words representing elements of
$S_n$.
\end{proposition}

\begin{proof}
Notice that an application of a rewriting rule from $\rel{R}$ strictly
decreases a word with respect to the lexicographic order. Hence
$(A,\rel{R})$ is acyclic. Since it does not alter the length of a
word, and since there are only finitely many words of a given length,
$(A,\rel{R})$ is globally finite. Hence $(A,\rel{R})$ is noetherian.

To see that $(A,\rel{R})$ is confluent, we prove that the irreducible
words are precisely those words that arise as left-to-right
postfix readings of binary search trees.

First, let us show that any word that arises as a
left-to-right-postfix reading of a binary search tree must be
irreducible. Let $w$ be a left-to-right postfix reading of some binary
search tree. That is, $w = LRP(\bst(w))$. Then by
\fullref{Proposition}{prop:lrpbstlexmin}, $w$ is the lexicographically
minimal word among all words representing the same element of $S_n$ as
$w$. Since an application of a rewriting rule from $\rel{R}$ always
decreases a word in the lexicographic order, it follows that no rule
in $\rel{R}$ can be applied to $w$. Thus $w$ is irreducible.

Now we prove that any irreducible word arises as a left-to-rise
postfix reading of a binary search tree.  We proceed by induction on
the length of the word. First, if $w \in A^*$ has length less than
$3$, no relation in $\rel{R}$ can be applied to $w$ and hence there is
no other word in $A^*$ representing the same element of $S_n$; thus $w
= LRP(\bst(w))$. This proves the base of the induction.

Now suppose that all irreducible words of length less than $k$ arise
from left-to-right postfix readings of binary search trees. Let $w$ be
an irreducible word of length $k$. Let $w = w'b$, where $b \in
A$. Since $w$ is irreducible, there is no left-hand side of a rule in
$\rel{R}$ in $w$, and thus there is no subword $ca$ in $w'$ such that
$a \leq b < c$. Therefore $w'$ factors as $w' = uv$, where every
symbol in $u$ is less than or equal to $b$ and every symbol in $v$ is
greater than $b$. The words $u$ and $v$ are irreducible and of length
less than $k$, and thus $u = LRP(T_u)$ and $v = LRP(T_v)$ for some
binary search trees $T_u$ and $T_v$. Form a new tree $T$ whose root is
labelled by $b$ and has left subtree $T_u$ and right subtree
$T_v$. Since every symbol in $u$ is less or equal to than $b$ and
every symbol in $v$ is greater than $v$, it follows that $T$ is a
binary search tree. Finally, by the definition of left-to-right
postfix reading, we have $LRP(T) = LRP(T_u)LRP(T_v)b = uvb = w$. This
completes the induction step.

Thus irreducible words are precisely those arising from left-to-right
postfix readings of binary search trees, which are precisely the
lexicographically minimal representatives of the elements of $S_n$ by
\fullref{Proposition}{prop:lrpbstlexmin}. Hence there is a unique
irreducible representative of each element of $S_n$, and so
$(A,\rel{R})$ is confluent.
\end{proof}

As a consequence of \fullref{Proposition}{prop:sylvesterfcrs}, $S_n$
admits a regular language of unique normal forms, namely the language
of words that do not include a left-hand side of a rule in $\rel{R}$:
\[
L = A^* - A^*\{cavb : a \leq b < c, v \in A^*\}A^*.
\]
Thus $L$ is the language of irreducible words of the rewriting system
$(A,\rel{R})$.

\subsection{Biautomaticity}
\label{subsec:sylvesterbiautomatic}

The aim of this section is to prove that $(A,L)$ is a biautomatic
structure for $S_n$. It is necessary to prove that ${}_\gamma L\lpad$,
${}_\gamma L\rpad$, $L_\gamma\lpad$, and $L_\gamma\rpad$ are regular.

\subsubsection{Left-multiplication}
\label{sec:sylvesterleft}

We begin by analyzing the reduction that can occur when we
left-multiply a normal form word by a single generator and then
rewrite back to a normal form word. Let $w \in L$ and $\gamma \in A$
and consider reducing $\gamma w$ to normal form.

\begin{lemma}
\label{lem:sylvesterlrrewrite}
After $k$ steps, the reduction of
$\gamma w$ must yield a word $u_k\gamma v_k$ such that
\begin{enumerate}
\item $w = u_kv_k$;
\item $\gamma$ is greater than every symbol of $u_k$;
\item if a rule $cavb \imreduces_{\rel{R}} acvb$ applies to $u_k\gamma
  v_k$, then $c$ is the distinguished symbol $\gamma$, the first
  symbol of $v_k$ is $a$, and some other symbol of $v_k$ is $b$.
\end{enumerate}
\end{lemma}

\begin{proof}
The proof is by induction on $k$. Let $u_0 = \emptyword$ and $v_0 =
w$. Then for $k=0$, conditions~1 and~2 hold immediately. Since $v_k =
w$ is irreducible, if any rule in $\rel{R}$ applies to $u_k\gamma v_k
= \gamma w$, it must apply as condition~3 specifies. This proves the
base case of the induction.

Now assume that the result holds for $k = \ell -1$; we aim to show it
holds for $k = \ell$. So after $\ell-1$ steps, reduction of $\gamma w$
yields $u_{\ell-1}\gamma v_{\ell-1}$ with conditions~1--3 being
satisfied for $k = \ell-1$. We will use the fact that conditions 1--3
hold for $k=\ell-1$ to prove conditions~1 and~2 for $k = \ell$;
condition~3 for $k=\ell$ then follows from conditions~1 and~2 for
$k=\ell$.

If $u_{\ell-1}\gamma v_{\ell-1}$ is irreducible, no further reduction
takes place and there is nothing more to prove. So suppose
$u_{\ell-1}\gamma v_{\ell-1}$ is not irreducible. Then by condition~3
for $k = \ell-1$ we can let $\alpha \in A$ and $v_{\ell} \in A^*$ be
such that $v_{\ell-1} = \alpha v_{\ell}$ and any rule in $\rel{R}$
that applies to $u_{\ell-1}\gamma v_{\ell-1}$ has $c = \gamma$, $a =
\alpha$, and $b$ being some symbol in $v_{\ell}$. Note that by the
definition of $\rel{R}$, we have $c > a$ and so $\gamma >
\alpha$. Applying this rule yields $u_{\ell-1}\alpha\gamma
v_{\ell}$. Let $u_\ell = u_{\ell-1}\alpha$.

By condition~1 for $k = \ell-1$, we have $w = u_{\ell-1}v_{\ell-1}$
and so $w = u_{\ell-1}\alpha v_\ell = u_\ell v_\ell$; this establishes
condition~1 for $k = \ell$.

By condition~2 for $k = \ell-1$, we know $\gamma$ is greater than
every symbol of $u_{\ell-1}$. Since $\gamma > \alpha$, we deduce that
$\gamma$ is greater than every symbol of $u_\ell$; this proves
condition~2 for $k=\ell$.

Finally, suppose some rule $cavb \imreduces_{\rel{R}} acvb$ (where $a
\leq b < c$) applies to $u_\ell\gamma v_\ell$. If $b$ lies in
$u_\ell$, then $cavb$ is a subword $u_\ell$ and thus of $w$, which
contradicts $w$ being irreducible. If $b$ is the distinguished letter
$\gamma$, then $c$ lies in $u_\ell$, which contradicts condition~2 for
$k = \ell$. Hence $b$ must be some symbol of $v_\ell$. If $ca$ is a
subword of either $u_\ell$ or $v_\ell$, then $cav'b$ is a subword of
$w = u_\ell v_\ell$ for some $v'$, contradicting the irreducibility of
$w$. If $a$ is the distinguished letter $\gamma$, then $c$ lies in
$u_k$, again contradicting the condition~2 for $k = \ell$. The only
remaining possibility is that $c$ is the distinguished symbol
$\gamma$, with $a$ being the first symbol of $v_k$. This proves
condition~3 for $k = \ell$.
\end{proof}

Let $w \in L$ and $\gamma \in A$. By
\fullref{Lemma}{lem:sylvesterlrrewrite}, rewriting $\gamma w$ to an
irreducible word consists of moving $\gamma$ to the right by applying
rules $cavb \imreduces_{\rel{R}} acvb$. Let us consider the symbols in
$\gamma w$ that play the role of $b$ in these rules. Let $\beta_1$ be
the first symbol playing this role. Apply rules involving $\beta_1$ as
many times as possibly, say $k_1$ times. This yields a word
$u_{k_1}\gamma v_{k_1}$. Let $\alpha_{k_1}$ be the first symbol of
$v_{k_1}$. Suppose this word $u_{k_1}\gamma v_{k_1}$ is not
irreducible. So some rule from $\rel{R}$ applies, with some symbol
$\beta_2 \neq \beta_1$ playing the role of $b$. Then $\alpha_{k_1} >
\beta_1$ (since otherwise a rule from $\rel{R}$ would apply with $c =
\gamma$, $a = \alpha_{k_1}$, and $b = \beta_1$). Since a rule applies with
$c = \gamma$, $a = \alpha_{k_1}$, and $b = \beta_2$, we have $\beta_2
\geq \alpha_{k_1} > \beta_1$. Apply rules involving $\beta_2$ as many
times as possible, say $k_2$ times, yielding $u_{k_2}\gamma
v_{k_2}$. Either this word is irreducible or, via the reasoning above,
it can be reduced by a rule in $\rel{R}$ with $b$ being $\beta_3 >
\beta_2$. Repeating this reasoning, we set a sequence $\beta_1 <
\beta_2 < \ldots$ which must terminate at some $\beta_\ell$ with an
irreducible word $u_{k_\ell}\gamma v_{k_\ell}$ since the alphabet $A$
is finite.

Let
\begin{align*}
H_b = \bigl\{(pcqrbs,pqcrbs):{}&p,r,s \in A^*, \\
&b,c \in A,\\
&c > b,\\
&q \in \{a \in A : a \leq b\}^+ \bigr\}.
\end{align*}
Clearly, $(pcqrbs,pqcrbs) \in H_b$ if and only if $pcqrbs$ reduces to
$pqcrbs$ using only rules from $\rel{R}$ applied to the distinguished
letters $c$ and $b$ with $a$ being the successive letters of
$q$. Hence, by the reasoning in the preceding paragraph, and using the
same notation,
\[
(\gamma w,u_{k_\ell}\gamma v_{k_\ell}) \in H_{\beta_1} \circ H_{\beta_2} \circ \cdots \circ H_{\beta_\ell}.
\]
Therefore,
\[
(w,x) \in {}_\gamma L \iff (\gamma w,x) \in \bigl((\gamma L) \times L\bigr) \cap \bigcup_{\ell=0}^n \;\;\bigcup_{\substack{\beta_1, \ldots,\beta_\ell\in A\\\beta_1 < \ldots <\beta_\ell}} H_{\beta_1} \circ \cdots \circ H_{\beta_\ell},
\]
or equivalently
\[
(\gamma,\emptyword){}_\gamma L = \bigl((\gamma L) \times L\bigr) \cap \bigcup_{\ell=0}^n \;\;\bigcup_{\substack{\beta_1, \ldots,\beta_\ell\in A\\\beta_1 < \ldots <\beta_\ell}} H_{\beta_1} \circ \cdots \circ H_{\beta_\ell}.
\]
(Note that the intersection with $(\gamma L) \times L$ is necessary
because $H_b$ also relates pairs of words that are not in this set.)

It is easy to see that $H_b$ is a rational relation, since a
transducer recognizing it only needs to store the symbol $c$ in its
state, check that the other symbols on the two tapes match, and that
the contents of the two tapes are of the required form. Hence
$(\gamma,\emptyword){}_\gamma L$ is a rational relation and so
${}_\gamma L$ is a rational relation. Since $(w,x) \in {}_\gamma L$
implies $|x| = |w| + 1$, it follows from
\fullref{Proposition}{prop:rationalbounded} that ${}_\gamma L\rpad$
and ${}_\gamma L\lpad$ are regular.

\subsubsection{Right-multiplication}
\label{sec:sylvesterright}

We now turn to right multiplication. By analogy with the hypoplactic
monoid, we will call any strictly decreasing word in $A^*$ a
\defterm{column}. Notice that since $A$ is finite, there are only
finitely many distinct columns. 

%% For any column $\alpha$, let $\alpha =
%% \alpha_{|\alpha|},\alpha_{|\alpha|-1},\ldots,\alpha_2,\alpha_1$, where
%% $\alpha_i \in A$ for all $i$. Notice that the order of the subscripts
%% of the symbols $\alpha_i$ matches the ordering of the symbols.

\begin{lemma}
\label{lem:norepeatedcolumns}
Let $w \in L$, and let $w = \alpha^{(1)}\cdots\alpha^{(k)}$ be the
decomposition of $w$ into maximal columns. If $\alpha^{(i)} =
\alpha^{(i+h)}$ with $h \geq 1$, then $|\alpha^{(i)}| = |\alpha^{(i+h)}| = 1$.
\end{lemma}

\begin{proof}
Suppose that $\alpha^{(i)} = \alpha^{(i+h)}$. Suppose, with the aim of
obtaining a contradiction, that $|\alpha^{(i)}| \geq 2$. Suppose
$\alpha^{(i)} = \alpha'ca$, where $c,a \in A$ and $\alpha' \in
A^*$. Let $b = a$. Then $c > a$ (since $\alpha^{(i)}$ is a decreasing
word) and so the condition $a \leq b < c$ holds. Thus the rewriting
rule $cavb \imreduces_{\rel{R}} acvb$ applies with $c,a$ being the
rightmost two symbols in $\alpha^{(i)}$ and $b$ being the rightmost
symbol in $\alpha^{(i+h)}$ (which is, by hypothesis, equal to
$\alpha^{(i)}$). Hence $w$ is not irreducible, which contradicts $w$
lying in $L$
\end{proof}

We will analyze the reduction that can occur when we right-multiply a
normal form word by a single generator and then rewrite back to a
normal form word. Let $w \in L$ and $\gamma \in A$ and consider
reducing $w\gamma$ to normal form.  Suppose $w = \alpha_1\cdots
\alpha_{|w|}$ for $\alpha_i \in A$. Let $G = \{i \in \{1,\ldots,|w|\} : \alpha_i \leq
\gamma\}$ and consider the word
\[
x = \Bigl[\prod_{\substack{1 \leq i \leq |w|\\i \in G}}\alpha_i\Bigr]\Bigl[\prod_{\substack{1 \leq j \leq |w|\\j \notin G}}\alpha_j\Bigr]\gamma.
\]

\begin{lemma}
\label{lem:wgammareduce}
The word $x$ is irreducible with respect to $\rel{R}$, and $w\gamma
\reduces_{\rel{R}} x$. Furthermore, rewriting of $w\gamma$ to $x$
only requires applying rules from $\rel{R}$ with $b = \gamma$.
\end{lemma}

\begin{proof}
First of all, notice that the rules in $\rel{R}$ apply with $b =
\gamma$ to move all letters less than or equal to $\gamma$ to the left
of those strictly greater than $\gamma$. So $w\gamma$ certainly
rewrites to $x$ in the given way; it remains to show that $x$ is
irreducible.

Suppose, with the aim of obtaining a contradiction, that some rule in
$\rel{R}$ applies to $x$. Then $x$
contains a subword $cavb$ for some $a \leq b < c$ and $v \in
A^*$.

Now, it is impossible to have $b = \gamma$, for this implies $a =
\alpha_i$ for some $i \in G$ and $c = \alpha_j$ for some $j \notin G$,
which in turn implies that $a$ appears to the left of $c$ in $x$,
which contradicts the form of the subword $cavb$.

It is also impossible to have $c = \alpha_i$ for $i \in G$ and $b =
\alpha_j$ for $j \notin G$, for then (by definition of $G$) $c \leq
\gamma < b$, contradicting the form of the subword $cavb$.

So $c$ and $b$ (and hence the whole subword $cavb$) must either both
lie within the product of the $\alpha_i$ with $i \in G$, or both lie
within the product of the $\alpha_j$ with $j \notin G$. In either
case, if the letters $ca$ were adjacent in the original word $w$, it
would not have been irreducible, since the rewriting of $w\gamma$ to
$x$ using $\rel{R}$ preserves the order in which the
$\alpha_i$ (with $i \in G$) appear and the order in which the
$\alpha_j$ (with $j \notin G$) appear.

Consider first the case that $cavb$ lies wholly within the product
$\alpha_i$ with $i \in G$. Then, since $ca$ were not adjacent in the
original word $w$, we have $c = \alpha_i$ and $a = \alpha_{i+h}$ with
$i,i+h \in G$ and $i+1,\ldots,i+h-1 \notin G$ for some $h \geq 2$. Let
$c' = \alpha_{i+h-1}$. Then, by definition of $G$, we have $a \leq b
< c'$ and so there is a left-hand side of a rule in $\rel{R}$ in
$w$, contradicting $w$ being irreducible.

Now consider the second case, where $cavb$ lies wholly within the product
$\alpha_j$ with $j \notin G$. Then, since $ca$ were not adjacent in the
original word $w$, we have $c = \alpha_j$ and $a = \alpha_{j+h}$ with
$j,j+h \notin G$ and $j+1,\ldots,j+h-1 \in G$ for some $h \geq 2$. Let
$a' = \alpha_{j+1}$. Then, by definition of $G$, we have $a' \leq b
< c$ and so there is a left-hand side of a rule in $\rel{R}$ in
$w$, contradicting $w$ being irreducible.

So each case leads to a contradiction. Hence $x$ is
irreducible.
\end{proof}

Recall that a word $w \in L$ admits a decomposition $w =
\alpha^{(1)}\alpha^{(2)}\cdots\alpha^{(k)}$ into maximal columns. Let
$I$ be the set of indices $i$ such that the letter $\alpha_i$ of $w$
is itself one of the columns $\alpha^{(j)}$. That is, $\alpha_i$ does
not lie in a column $\alpha^{(j)}$ containing two or more symbols from
$A$. Since the length of columns is bounded by $n$, there is a bounded
number of possible columns of length at least $2$. Each of these
columns appears at most once in $w$ by
\fullref{Lemma}{lem:norepeatedcolumns}. Since each of these columns
has length at most $n$, there is a bound $M$ (dependent only on $n$)
on the number of indices not in $I$.

\begin{lemma}
For all $i,i+h \in I$, we have $\alpha_i \leq \alpha_{i+h}$.
\end{lemma}

\begin{proof}
Suppose, with the aim of obtaining a contradiction, that for some
$i,i+h \in I$, we have $\alpha_i > \alpha_{i+h}$. Since $\alpha_{i+h}$
is a column (a maximal strictly decreasing subword of), we have $\alpha_{i+h-1}
\leq \alpha_{i+h}$. In particular, $h \geq 2$. Since $\alpha_i >
\alpha_{i+h} \geq \alpha_{i+h-1}$, the sequence
$\alpha_i,\ldots,\alpha_{i+h-1}$ starts greater than $\alpha_{i+h}$
and ends less than or equal to $\alpha_{i+h}$. So there must be some
$g \in \{0,\ldots,h-2\}$ such that $\alpha_{i+g} > \alpha_{i+h} \geq
\alpha_{i+g+1}$. Hence a rule from $\rel{R}$ applies to $w$ with $c =
\alpha_{i+g}$, $a=\alpha_{i+g+1}$, and $b = \alpha_{i+h}$, which
contradicts $w$ being irreducible.
\end{proof}

Thus, the rewriting of $w\gamma$ to $x$ does not alter
the relative positions of symbols $\alpha_i$ with $i \in I$. So only
symbols with subscripts not in $I$ have to be moved rightwards using
rules in $\rel{R}$ with $b= \gamma$ in order to rewrite $w\gamma$ to
$x$. That is, at most $M$ symbols, each greater than $\gamma$, must be
moved to the right of all symbols less or equal to $\gamma$ (excepting
$\gamma$ itself).

Since the symbol $\gamma$ is not moved during this rewriting, we will
consider the relation
\[
L'_\gamma = \{(w,x') : w \in L, w \reduces_{\rel{R}} x'\gamma \in L\},
\]
Notice that $L_\gamma = L'_\gamma(\emptyword,\gamma)$, and so if a
synchronous transducer recognizes $L'_\gamma$, it recognizes $L_\gamma$. 

Let
\begin{align*}
J_\gamma = \bigl\{(pcqr,pqcr):{}&p \in A^* \\
&q \in \{a \in A : a \leq \gamma\}^+ \\
&r \in \{a \in A : a > \gamma\}^* \\
&c \in A, c > \gamma\bigr\}.
\end{align*}
Notice that if $(pcqr,pqcr) \in J_\gamma$, then $pqcr\gamma$ is the
word obtained from $pcqr\gamma$ by applying rewriting rules $cavb
\imreduces_\rel{R} acvb$ with $\gamma$ being $b$ and successive
letters from $q$ being $a$. Thus if we start with our word $w$ and
apply $J_\gamma$, we move the rightmost letter $c$ that is greater
than $\gamma$ but which lies to the left of some letter less than or
equal to $\gamma$ into its proper place. Iterating this process will
therefore yield $x'$. Since there are at most $M$ symbols that have to
be moved rightwards to their proper places to obtain $x'$ from $w$, at
most $M$ iterations are required. Therefore
\[
L'_\gamma = (L \times (L/\gamma)) \cap \bigcup_{\ell = 1}^M \underbrace{J_\gamma\circ\cdots\circ J_\gamma}_{\text{$\ell$ times}}
\]
(Recall that $L/\gamma = \{w \in A^* : w\gamma \in L\}$. Note that we have to
take the intersection with $L \times (L/\gamma)$ because
$J_\gamma\circ\cdots\circ J_\gamma$ may contains pairs $(y,z)$ where
$y$ and $z\gamma$ are not irreducible.)

It is easy to see that $J_\gamma$ is a rational relation, since a
transducer recognizing it only needs to store the symbol $c$ in its
state, check that the other symbols on the two tapes match, and that
the contents of the two tapes are of the required form. Hence
$L'_\gamma$ is a rational relation and so $L_\gamma$ is a rational
relation. Since $(w,x) \in L_\gamma$ implies $|x| = |w| + 1$, it
follows from \fullref{Proposition}{prop:rationalbounded} that
$L_\gamma\rpad$ and $L_\gamma\lpad$ are regular.

In the previous section, we proved that ${}_\gamma L\rpad$ and
${}_\gamma L\lpad$ are regular. Thus we have proved:

\begin{theorem}
$(A,L)$ is a biautomatic structure for the sylvester monoid $S_n$.
\end{theorem}

\bibliography{automaticsemigroups,languages,presentations,semigroups,\jobname,c_publications}
\bibliographystyle{alphaabbrv}

\end{document}